\documentclass[12pt]{amsart}
\usepackage{a4wide}
\usepackage{graphicx}
\usepackage{color}
\usepackage{cancel}
\allowdisplaybreaks

\let\pa\partial  
\let\na\nabla  
  
\newcommand{\N}{{\mathbb N}}  
\newcommand{\R}{{\mathbb R}} 
 
\newcommand{\diver}{\operatorname{div}}  
\newcommand{\T}{{\mathbb T}}

\newcommand{\TT}{{\mathcal T}}
\newcommand{\E}{{\mathcal E}}
\newcommand{\dd}{\text{\rm d}}
\newcommand{\m}{\text{\rm m}}
\newcommand{\D}{{\mathcal D}}

\newtheorem{theorem}{Theorem}   
\newtheorem{lemma}[theorem]{Lemma}   
\newtheorem{proposition}[theorem]{Proposition}   
\newtheorem{remark}[theorem]{Remark}   
\newtheorem{corollary}[theorem]{Corollary}

\begin{document}  

\title[Entropy dissipative finite-volume scheme]{
Entropy-dissipative discretization \\
of nonlinear diffusion equations \\
and discrete Beckner inequalities}

\author[C. Chainais-Hillairet]{Claire Chainais-Hillairet}
\address{Laboratoire Paul Painlev\'e, U.M.R. CNRS 8524,
Universit\'e Lille 1 Sciences et Technologies, Cit\'e Scientifique,
59655 Villeneuve d'Ascq Cedex and 
Project-Team SIMPAF, INRIA Lille-Nord-Europe, Villeneuve d'Ascq, France}
\email{Claire.Chainais@math.univ-lille1.fr}

\author[A. J\"ungel]{Ansgar J\"ungel}
\address{Institute for Analysis and Scientific Computing, Vienna University of  
	Technology, Wiedner Hauptstra\ss e 8--10, 1040 Wien, Austria}
\email{juengel@tuwien.ac.at} 

\author[S. Schuchnigg]{Stefan Schuchnigg}
\address{Institute for Analysis and Scientific Computing, Vienna University of  
	Technology, Wiedner Hauptstra\ss e 8--10, 1040 Wien, Austria}
\email{stefan.schuchnigg@tuwien.ac.at} 

\date{\today}

\thanks{The authors have been partially supported by the Austrian-French Project 
Amade\'e of the Austrian Exchange Service (\"OAD).
The second and last author acknowledge partial support from   
the Austrian Science Fund (FWF), grants P22108, P24304, I395, and W1245.
Part of this work was written during the stay of the second author at the
University of Technology, Munich (Germany), as a John von Neumann
Professor. The second author thanks the Department of Mathematics for the hospitality
and Jean Dolbeault (Paris) for very helpful discussions on Beckner inequalities} 

\begin{abstract}
The time decay of fully discrete finite-volume approximations of porous-medium
and fast-diffusion equations with Neumann or periodic boundary conditions is proved
in the entropy sense.
The algebraic or exponential decay rates are computed explicitly.
In particular, the numerical scheme dissipates all zeroth-order entropies
which are dissipated by the continuous equation. The proofs are based on
novel continuous and discrete generalized Beckner inequalities. 
Furthermore, the exponential decay of some first-order entropies is proved 
in the continuous and discrete case using systematic integration by parts. 
Numerical experiments in one and two
space dimensions illustrate the theoretical results and indicate that some
restrictions on the parameters seem to be only technical.
\end{abstract}

 

\maketitle


\section{Introduction}

This paper is concerned with the time decay of fully discrete finite-volume
solutions to the nonlinear diffusion equation
\begin{equation}\label{1.pme}
  u_t = \Delta(u^\beta) \quad\mbox{in }\Omega,\ t>0, \quad
  u(\cdot,0)=u_0\quad\mbox{in }\Omega,
\end{equation}
and with the relation to discrete generalized Beckner inequalities.
Here, $\beta>0$ and $\Omega\subset\R^d$ ($d\ge 1$) is a bounded domain.
When $\beta>1$, \eqref{1.pme} is called the porous-medium equation, 
describing the flow of an isentropic gas through a porous medium \cite{Vaz07}. 
Equation \eqref{1.pme} with $\beta<1$ is referred to as the fast-diffusion equation,
which appears, for instance, in plasma physics with $\beta=\frac12$ \cite{Ber79} or
in semiconductor theory 
with $0<\beta<1$ \cite{Kin88}. 
We impose homogeneous Neumann boundary conditions
\begin{equation}\label{1.bc}
  \na(u^\beta)\cdot\nu = 0\quad\mbox{on }\pa\Omega,\ t>0,
\end{equation}
where $\nu$ denotes the unit normal exterior vector to $\pa\Omega$,
or multiperiodic boundary conditions (i.e.\ $\Omega$ equals the torus $\T^d$). 
Let us denote by {\rm m} the Lebesgue measure in $\R^d$ or $\R^{d-1}$; we assume for simplicity that $\mbox{m}(\Omega)=1$.
For existence and uniqueness results for the porous-medium equation 
in the whole space or under suitable boundary conditions, 
we refer to the monograph \cite{Vaz07}.

In the literature, there exist many numerical schemes for nonlinear diffusion
equations related to \eqref{1.pme}. Numerical techniques include (mixed) 
finite-element methods \cite{AWZ96,EbLi08,Ros83}, 
finite-volume approximations \cite{EGH98,Ohl01},
high-order relaxation ENO-WENO schemes \cite{CNPS07}, 
or particle methods \cite{LiMa01}. In these references, also stability and
numerical convergence properties are proved. 

The preservation of the structure of diffusion equations is a 
very important property of a numerical scheme. For instance, ideas employed
for hyperbolic conservation laws were extended to degenerate diffusion equations,
like the porous-medium equation,
which may behave like hyperbolic ones in the regions of degeneracy \cite{MiVo03}.
Positivity-preserving schemes for nonlinear fourth-order equations were
thoroughly investigated in the context of lubrication-type equations
\cite{BBG98,ZhBe00} and quantum diffusion equations \cite{JuPi01}. 
Entropy-consistent finite-volume finite-element schemes for the fourth-order
thin-film equation were suggested by Gr\"un and Rumpf \cite{GrRu00}. 
For quantum diffusion models, an entropy-dissipative relaxation-type 
finite-difference discretization was investigated by Carrillo et al.\ \cite{CJT03}. Furthermore, entropy-dissipative schemes for electro-reaction-diffusion systems 
were derived by Glitzky and G\"artner \cite{GlGa09}. 
However, it seems that there does not exist any systematic study on entropy-dissipative
discretizations for \eqref{1.pme} and the time decay of their discrete solutions.

Our first aim is to prove that the finite-volume scheme for \eqref{1.pme}-\eqref{1.bc},
defined in \eqref{2.fv}, dissipates the discrete versions of the functionals
\begin{align}
  E_\alpha[u] &= \frac{1}{\alpha+1}\left(\int_\Omega u^{\alpha+1}dx
  - \Big(\int_\Omega udx\Big)^{\alpha+1}\right), \label{1.E} \\
  F_\alpha[u] &= \frac{1}{2}\int_\Omega|\na u^{\alpha/2}|^2 dx, \quad \alpha>0.
  \label{1.F}
\end{align}
In fact, we will prove (algebraic or exponential) convergence
rates at which the discrete functionals converge to zero as $t\to\infty$. 
We call $E_\alpha$ a zeroth-order entropy and $F_\alpha$
a first-order entropy. The functional 
$F_{1}$ is known as the Fisher information, used in mathematical statistics and
information theory \cite{Edg08}. Our analysis of the decay rates of the entropies
will be guided by the entropy-dissipation method. An essential ingredient of
this technique is a functional inequality relating the entropy to the entropy 
dissipation \cite{AMTU01,CJMTU01}. For the diffusion equation \eqref{1.pme},
this relation is realized by the Beckner inequality \cite{Bec89}.

The entropy-dissipation method was applied to \eqref{1.pme} in the whole space
to prove the decay of the solutions to the asymptotic self-similar profile
in, e.g., \cite{CaTo00,DeDo02}.  
The convergence towards the constant steady state on the
one-dimensional torus was proved in \cite{CDGJ06}. 
However, we are not aware of general entropy decay estimates for solutions 
to \eqref{1.pme} to the constant steady state, even in the continuous case. 
The reason might be that generalizations to the Beckner inequality, 
needed to relate the entropy dissipation to
the entropy, are missing. As our second aim, we prove these generalized
Beckner inequalities and provide some decay estimates for $E_\alpha$ and $F_\alpha$
along trajectories of \eqref{1.pme}.

This paper splits into two parts. The first part is concerned
with the proof of generalized Beckner inequalities and the decay rates
for the continuous case. The second---and main---part is the ``translation''
of these results to an implicit Euler finite-volume discretization of \eqref{1.pme}. 
In the following, we summarize our main results.

The {\em first result} is the proof of the generalized Beckner inequality
\begin{equation}\label{1.bec}
  \int_\Omega|f|^q dx - \left(\int_\Omega|f|^{1/p}dx\right)^{pq} 
  \le C_B(p,q)\|\na f\|_{L^2(\Omega)}^q,
\end{equation}
where $f\in H^1(\Omega)$ and $0<q\le 2$, $pq\ge 1$. In the case $q=2$,
we require that $\frac12-\frac{1}{d}\le p\le 1$. The constant $C_B(p,q)>0$
only depends on $p$, $q$, and the constant of the Poincar\'e-Wirtinger
inequality (see Lemma \ref{lem.beckner} for details). The usual Beckner
inequality \cite{Bec89} is recovered for $q=2$; see Remark \ref{rem.bec} 
for a comparison of related Beckner inequalities in the literature.
The proof is elementary and only employs the Poincar\'e-Wirtinger inequality. 
By using a discrete version of this inequality (see \cite{BCF12}),
the proof can be easily ``translated'' to derive the discrete 
generalized Beckner inequality
$$
  \int_\Omega|f|^q dx - \left(\int_\Omega|f|^{1/p}dx\right)^{pq} 
  \le C_b(p,q)|f|_{1,2,\TT}^q,
$$
where $f$ is a function which is constant on each cell of the finite-volume
triangulation $\TT$ of $\Omega$ and $|\cdot|_{1,2,\TT}$ is the discrete
$H^1$-seminorm; see Section \ref{sec.not} and Lemma \ref{lem.d.bec} for details. 

The {\em second result} is the time decay of the entropies $E_\alpha$
and $F_\alpha$ along trajectories of \eqref{1.pme}. Differentiating $E_\alpha[u(t)]$
with respect to time and employing the Beckner inequality \eqref{1.bec}, we show
for $\beta>1$ that
$$
  \frac{dE_\alpha}{dt}[u(t)] \le CE_\alpha[u(t)]^{(\alpha+\beta)/(\alpha+1)},
  \quad t>0,
$$
where $C>0$ depends on $\alpha$, $\beta$, and $C_B(p,q)$. By a nonlinear
Gronwall inequality, this implies the algebraic decay of $u(t)$ to equilibrium
in the entropy sense; see Theorem \ref{thm.E.poly}. 
If the solution is positive and $0<\alpha\le 1$, the above inequality becomes
$$
  \frac{dE_\alpha}{dt}[u(t)] \le C(u_0)E_\alpha[u(t)],
  \quad t>0,
$$
which results in an exponential decay rate; see Theorem \ref{thm.E.exp}.

The first-order entropies $F_\alpha[u(t)]$ decay exponentially fast (for
positive solutions) for all $(\alpha,\beta)$ lying in the strip
$-2\le\alpha-2\beta\le 1$ (one-dimensional case) or in the region
$M_d$, which is illustrated in Figure \ref{fig.ab} below (multi-dimensional case);
see Theorems \ref{thm.F.exp} and \ref{thm.F.exp3}. The proof is based on
systematic integration by parts \cite{JuMa06}. In order to avoid
boundary integrals arising from the iterated integrations by parts, these
results are valid only if $\Omega=\T^d$. Notice that all these results are new.

The {\em third---and main---result} is the ``translation'' of the continuous decay rates
to the finite-volume approximation. We obtain the same results for a discrete
version of $E_\alpha$ in Theorems \ref{thm.E.poly2} (algebraic decay)
and \ref{thm.E.exp2} (exponential decay). The situation is different for
the first-order entropies $F_\alpha$. The reason is that it is very difficult
to ``translate'' the iterated integrations by parts to iterated summations
by parts since there is no discrete nonlinear chain rule. 
For the zeroth-order entropies, this is done by exploiting the
convexity of the mapping $x\mapsto x^{\alpha+1}$. 
For the first-order entropies, we employ the concavity of the discrete version 
of $dF_\alpha/dt$ with respect to the time approximation parameter. 
We prove in Theorem \ref{thm.F.exp2} that for $1\le\alpha\le 2$ and $\beta=\alpha/2$,
the discrete first-order entropy is monotone (multi-dimensional case) and decays
exponentially fast (one-dimensional case). We stress the fact that this
is the first result in the literature on the decay of discrete first-order entropies.

Throughout this paper, we assume that the solutions to \eqref{1.pme} are
smooth and positive such that we can perform all the computations and
integrations by parts. In particular, we avoid any technicalities due to
the degeneracy ($\beta>1$) or singularity ($\beta<1$) in \eqref{1.pme}.
Most of our results can be generalized to nonnegative weak solutions by
using a suitable approximation scheme but details are left to the reader.

The paper is organized as follows. In Section \ref{sec.cont}, we investigate
the continuous case. We prove two novel generalized Beckner inequalities
in Section \ref{sec.cont.bec}, the algebraic and exponential decay of
$E_\alpha[u]$ in Section \ref{sec.cont.E}, and the exponential decay of
$F_\alpha[u]$ in Section \ref{sec.cont.F}. The discrete situation is analyzed
in Section \ref{sec.disc}. After introducing the finite-volume 
scheme in Section \ref{sec.not},
the algebraic and exponential decay rates for the discrete version of $E_\alpha[u]$
is shown in Section \ref{sec.disc.E}, and the exponential decay of the discrete
version of $F_\alpha[u]$ is proved in Section \ref{sec.disc.F}.
In Section \ref{sec.num}, we illustrate the theoretical results by numerical
experiments in one and two space dimensions. They indicate that some of the
restrictions on the parameters $(\alpha,\beta)$ seem to be only technical.
In the appendix, a discrete nonlinear Gronwall lemma and some auxiliary inequalities
are proved.


\section{The continuous case}\label{sec.cont}

It is convenient to analyze first the continuous case before extending the
ideas to the discrete situation. 
We prove new convex Sobolev inequalities and algebraic and exponential decay rates 
of the solutions to \eqref{1.pme}.

\subsection{Generalized Beckner inequalities}\label{sec.cont.bec}

We assume in this subsection that $\Omega\subset\R^d$ ($d\ge 1$) is a bounded domain
such that the Poincar\'e-Wirtinger inequality
\begin{equation}\label{b.pwi}
  \|f-\bar f\|_{L^2(\Omega)}\le C_P\|\na f\|_{L^2(\Omega)}
\end{equation}
for all $f\in H^1(\Omega)$
holds, where $\bar f= \mbox{m}(\Omega)^{-1}\int_\Omega fdx$ and $C_P>0$
only depends on $d$ and $\Omega$.
This is the case if, for instance, $\Omega$ has the cone property 
\cite[Theorem 8.11]{LiLo10} or if $\pa\Omega$ is locally Lipschitz continuous
\cite[Theorem 1.3.4]{WYW06}. Suppose that $\mbox{\rm m}(\Omega)=1$ (to shorten the proof).
Before stating our main result, we prove the following lemma.

\begin{lemma}[Generalized Poincar\'e-Wirtinger inequality]\label{lem.gpw}
Let $0<q\le 2$ and $f\in H^1(\Omega)$. Then 
\begin{equation}\label{b.gpw}
  \|f\|_{L^2(\Omega)}^q \le C_P^q\|\na f\|_{L^2(\Omega)}^q + \|f\|_{L^q(\Omega)}^q
\end{equation}
holds, where $C_P>0$ is the constant of the Poincar\'e-Wirtinger inequality 
\eqref{b.pwi}.
\end{lemma}

\begin{proof}
Let first $1\le q\le 2$. The Poincar\'e-Wirtinger inequality \eqref{b.pwi}
\begin{equation}\label{b.pwi2}
  \|f\|_{L^2(\Omega)}^2 - \|f\|_{L^1(\Omega)}^2
  = \|f-\bar f\|_{L^2(\Omega)}^2 \le C_P^2\|\na f\|_{L^2(\Omega)}^2
\end{equation}
together with the H\"older inequality leads to
\begin{equation}\label{b.aux2}
  \|f\|_{L^2(\Omega)}^2 \le C_P^2\|\na f\|_{L^2(\Omega)}^2
  + \|f\|_{L^q(\Omega)}^2.
\end{equation}
Here we use the assumption $\mbox{m}(\Omega)=1$. 
Since $q/2\le 1$, it follows that
$$
  \|f\|_{L^2(\Omega)}^q 
  \le \big(C_P^2\|\na f\|_{L^2(\Omega)}^2 + \|f\|_{L^q(\Omega)}^2\big)^{q/2}
  \le C_P^q\|\na f\|_{L^2(\Omega)}^q + \|f\|_{L^q(\Omega)}^q,
$$
which equals \eqref{b.gpw}.

Next, let $0<q<1$. We claim that
\begin{equation}\label{b.ab}
  a^{q/2} - a^{q-1}b^{1-q/2} \le (a-b)^{q/2} \quad\mbox{for all }a\ge b>0.
\end{equation}
Indeed, setting $c=b/a$, this inequality is equivalent to
$$
  1-c^{1-q/2} \le (1-c)^{q/2} \quad\mbox{for all }0<c\le 1.
$$
The function $g(c)=1-c^{1-q/2}-(1-c)^{q/2}$ for $c\in[0,1]$ satisfies
$g(0)=g(1)=0$ and $g''(c)=(q/2)(1-q/2)(c^{-1-q/2}+(1-c)^{q/2-2})\ge 0$ for
$c\in(0,1)$, which implies that $g(c)\le 0$, proving \eqref{b.ab}.
Applying \eqref{b.ab} to $a=\|f\|_{L^2(\Omega)}^2$ and $b=\|f\|_{L^1(\Omega)}^2$
and using \eqref{b.pwi2}, we find that
\begin{equation}\label{b.ab2}
  \|f\|_{L^2(\Omega)}^q - \|f\|_{L^2(\Omega)}^{2(q-1)}\|f\|_{L^1(\Omega)}^{2-q}
  \le \big(\|f\|_{L^2(\Omega)}^2-\|f\|_{L^1(\Omega)}^2\big)^{q/2}
  \le C_P^q\|\na f\|_{L^2(\Omega)}^q.
\end{equation}
In order to get rid of the $L^1$ norm, we employ the interpolation inequality
\begin{equation}\label{b.aux4}
  \|f\|_{L^1(\Omega)} = \int_\Omega|f|^\theta|f|^{1-\theta}dx
  \le \|f\|_{L^q(\Omega)}^\theta\|f\|_{L^2(\Omega)}^{1-\theta},
\end{equation}
where $\theta=q/(2-q)<1$. Since $(2-q)\theta=q$ and $(2-q)(1-\theta)=2(1-q)$,
\eqref{b.ab2} becomes
$$
  \|f\|_{L^2(\Omega)}^q - \|f\|_{L^q(\Omega)}^q 
  \le C_P^q\|\na f\|_{L^2(\Omega)}^q,
$$
which is the desired inequality.
\end{proof}

\begin{lemma}[Generalized Beckner inequality I]\label{lem.beckner}
Let $d\ge 1$, $0<q<2$, $pq\ge 1$ or $q=2$, $\frac12-\frac{1}{d}\le p\le 1$
$(0<p\le 1$ if $d\le 2)$, 
and let $f\in H^1(\Omega)$. Then the generalized Beckner inequality
\begin{equation}\label{b.beckner}
  \int_\Omega |f|^q dx - \left(\int_\Omega|f|^{1/p}dx\right)^{pq}
  \le C_B(p,q)\|\na f\|_{L^2(\Omega)}^q
\end{equation}
holds, where 
$$
  C_B(p,q) = \frac{2(pq-1)C_P^q}{2-q}\quad\mbox{if }q<2, \quad
  C_B(p,2) = C_P^2\quad\mbox{if }q=2,
$$
and $C_P>0$ is the constant of the Poincar\'e-Wirtinger inequality \eqref{b.pwi}.
\end{lemma}

\begin{remark}\label{rem.bec}\rm
The case $q=2$ corresponds to the usual Beckner inequality \cite{Bec89}
$$
  \int_\Omega |f|^2 dx - \left(\int_\Omega|f|^{2/r}dx\right)^{r}
  \le C_B(p,2)\|\na f\|_{L^2(\Omega)}^2, 
$$
where $1\le r=2p\le 2$. It is shown in \cite{DNS08} that the constant $C_B(p,2)$
can be related to the lowest positive eigenvalue of a Schr\"odinger operator if
$\Omega$ is convex.
On the one-dimensonal torus, the generalized Beckner inequality \eqref{b.beckner} 
for $p>0$ and $0<q<2$ has been derived in \cite{CDGJ06}. 
In the multi-dimensional situation, the special case $p=2/q$ was 
proved in \cite{DGGW08}. In this work, it was also shown that
\eqref{b.beckner} with $q>2$ and $p=2/q$ cannot be true unless the Lebesgue measure
$dx$ is replaced by the Dirac measure.
In the limit $pq\to 1$, \eqref{b.beckner} leads to a generalized logarithmic Sobolev
inequality (see \eqref{b.lsi} below). If $q=2$ in this limit, the usual
logarithmic Sobolev inequality \cite{Gro75} is obtained.
\qed
\end{remark}

\begin{proof}[Proof of Lemma \ref{lem.beckner}]
Let first $q=2$. Then the Beckner inequality is a consequence
of the Poincar\'e-Wirtinger inequality \eqref{b.pwi} and the Jensen inequality:
$$
  C_P^2\|\na f\|_{L^2(\Omega)}^2 \ge \|f-\bar f\|_{L^2(\Omega)}^2
  = \|f\|_{L^2(\Omega)}^2 - \|f\|_{L^1(\Omega)}^2
  \ge  \int_\Omega f^2 dx - \left(\int_\Omega |f|^{2/r}dx\right)^r,
$$
where $1-\frac{2}{d}\le r\le 2$ ($0<r\le 2$ if $d\le 2$). 
The lower bound for $r$ ensures that the embedding
$H^1(\Omega)\hookrightarrow L^{2/r}(\Omega)$ is continuous.
The choice $p=r/2\in[\frac12-\frac{1}{d},1]$ yields the formulation \eqref{b.beckner}.

Next, let $0<q<2$.
The first part of the proof is inspired by the proof of
Proposition 2.2 in \cite{DGGW08}. 
Taking the logarithm of the interpolation inequality
$$
  \|f\|_{L^r(\Omega)} \le \|f\|_{L^q(\Omega)}^{\theta(r)}
  \|f\|_{L^2(\Omega)}^{1-\theta(r)},
$$
where $q\le r\le 2$ and $\theta(r)=q(2-r)/(r(2-q))$, gives
$$
  F(r) := \frac{1}{r}\log\int_\Omega |f|^r dx 
  - \frac{\theta(r)}{q}\log\int_\Omega |f|^q dx
  - \frac{1-\theta(r)}{2}\log\int_\Omega |f|^2 dx \le 0.
$$
The function $F:[q,2]\to\R$ is nonpositive, differentiable and $F(q)=0$. Therefore,
$F'(q)\le 0$, which equals
\begin{align*}
  -\frac{1}{q^2}\log\int_\Omega|f|^q dx 
  &+ \frac{1}{q}\left(\int_\Omega|f|^q dx\right)^{-1}\int_\Omega |f|^q\log|f|dx \\
  &{}+ \theta'(q)\left(\frac{1}{2}\log\int_\Omega |f|^2 dx 
  - \frac{1}{q}\log\int_\Omega |f|^q dx\right) \le 0.
\end{align*}
We multiply this inequality by $q^2\int_\Omega|f|^q dx$ to obtain
\begin{equation}\label{b.aux}
  \int_\Omega |f|^q\log\frac{|f|^q}{\|f\|_{L^q(\Omega)}^q}dx
  \le \frac{2}{2-q}\|f\|_{L^q(\Omega)}^q\log
  \frac{\|f\|_{L^2(\Omega)}^q}{\|f\|_{L^q(\Omega)}^q}.
\end{equation}
Then, we employ Lemma \ref{lem.gpw} and the inequality $\log(x+1)\le x$ for
$x\ge 0$ to infer that
$$
  \|f\|_{L^q(\Omega)}^q\log\frac{\|f\|_{L^2(\Omega)}^q}{\|f\|_{L^q(\Omega)}^q}
  \le \|f\|_{L^q(\Omega)}^q\log\left(\frac{C_P^q 
  \|\na f\|_{L^2(\Omega)}^q}{\|f\|_{L^q(\Omega)}^q} + 1\right) 
  \le C_P^q \|\na f\|_{L^2(\Omega)}^q.
$$
Combining this inequality and \eqref{b.aux}, we conclude the generalized
logarithmic Sobolev inequality 
\begin{equation}\label{b.lsi}
  \int_\Omega |f|^q\log\frac{|f|^q}{\|f\|_{L^q(\Omega)}^q}dx
  \le \frac{2C_P^q}{2-q}\|\na f\|_{L^2(\Omega)}^q.
\end{equation}

The generalized Beckner inequality \eqref{b.beckner} 
is derived by extending slightly the 
proof of \cite[Corollary 1]{LaOl00}. Let
$$
  G(r) = r\log\int_\Omega |f|^{q/r} dx, \quad r\ge 1.
$$
The function $G$ is twice differentiable with
\begin{align*}
  G'(r) &= \left(\int_\Omega |f|^{q/r}dx\right)^{-1}\left(
  \int_\Omega |f|^{q/r}dx\log\int_\Omega|f|^{q/r}dx 
  - \frac{q}{r}\int_\Omega|f|^{q/r}\log|f|dx\right), \\
  G''(r) &= \frac{q^2}{r^3}\left(\int_\Omega |f|^{q/r}dx\right)^{-2}
  \left(\int_\Omega |f|^{q/r}dx\int_\Omega|f|^{q/r}(\log|f|)^2dx
  - \left(\int_\Omega|f|^{q/r}\log|f|dx\right)^2\right).
\end{align*}
The Cauchy-Schwarz inequality shows that $G''(r)\ge 0$, i.e., $G$ is convex.
Consequently, $r\mapsto e^{G(r)}$ is also convex and
$r\mapsto H(r) = -(e^{G(r)}-e^{G(1)})/(r-1)$ is nonincreasing on $(1,\infty)$, 
which implies that
$$
  H(r) \le \lim_{t\to 1} H(t) = -e^{G(1)}G'(1) 
  = \int_\Omega |f|^q\log\frac{|f|^q}{\|f\|_{L^q(\Omega)}^q}dx.
$$
This inequality is equivalent to
\begin{equation}\label{b.aux3}
  \frac{1}{r-1}\left(\int_\Omega|f|^{q}dx 
  - \left(\int_\Omega|f|^{q/r}dx\right)^r \right)
  \le \int_\Omega |f|^q\log\frac{|f|^q}{\|f\|_{L^q(\Omega)}^q}dx.
\end{equation}
Combining this inequality and the generalized logarithmic Sobolev inequality
\eqref{b.lsi}, it follows that
$$
  \int_\Omega|f|^{q}dx - \left(\int_\Omega|f|^{q/r}dx\right)^r
  \le \frac{2(r-1)C_P^q}{2-q}\|\na f\|_{L^2(\Omega)}^q
$$
for all $0<q<2$ and $r\ge 1$. Setting $p:=r/q$, this proves \eqref{b.beckner}
for all $pq=r\ge 1$.
\end{proof}

For the proof of exponential decay rates, we need the following variant of
the Beckner inequality.

\begin{lemma}[Generalized Beckner inequality II]\label{lem.beckner2}
Let $0<q<2$, $pq\ge 1$ and $f\in H^1(\Omega)$. Then
\begin{equation}\label{b.beckner2}
  \|f\|_{L^q(\Omega)}^{2-q}
  \left(\int_\Omega|f|^q dx - \left(\int_\Omega|f|^{1/p}dx\right)^{pq}\right)
  \le C_B'(p,q)\|\na f\|_{L^2(\Omega)}^2,
\end{equation}
where 
$$
  C_B'(p,q) = \left\{\begin{array}{ll}
  \dfrac{q(pq-1)C_P^2}{2-q} &\mbox{ if }1\le q<2, \\[3mm]
  (pq-1)C_P^2 &\mbox{ if }0<q<1.
  \end{array}\right.
$$
\end{lemma}

\begin{proof}
By \eqref{b.aux}, it holds that for all $0<q<2$,
$$
  \int_\Omega|f|^q\log\frac{|f|^q}{\|f\|_{L^q(\Omega)}^q}dx
  \le \frac{q}{2-q}\|f\|_{L^q(\Omega)}^q\log
  \frac{\|f\|_{L^2(\Omega)}^2}{\|f\|_{L^q(\Omega)}^2}.
$$
Then, for $q>1$, the Poincar\'e-Wirtinger inequality in the version 
\eqref{b.aux2} and the inequality $\log(x+1)\le x$ for $x\ge 0$ yield
\begin{equation}\label{b.aux5}
  \|f\|_{L^q(\Omega)}^q\log
  \frac{\|f\|_{L^2(\Omega)}^2}{\|f\|_{L^q(\Omega)}^2}
  \le \|f\|_{L^q(\Omega)}^q\log\left(
  C_P^2\frac{\|\na f\|_{L^2(\Omega)}^2}{\|f\|_{L^q(\Omega)}^2} + 1\right)
  \le C_P^2\|f\|_{L^q(\Omega)}^{q-2}\|\na f\|_{L^2(\Omega)}^2.
\end{equation}
Taking into account \eqref{b.aux3}, the conclusion follows for $q>1$.

Let $0<q\le 1$. Suppose that the following inequality holds:
\begin{equation}\label{b.wirt}
  \|f\|_{L^q(\Omega)}^2 + \frac{2-q}{q}C_P^2\|\na f\|_{L^2(\Omega)}^2
  - \|f\|_{L^2(\Omega)}^2 \ge 0.
\end{equation}
This implies that, by \eqref{b.aux3} and for $r=pq$,
\begin{align*}
  \int_\Omega|f|^q dx - \left(\int_\Omega|f|^{q/r}dx\right)^r
  &\le \frac{(pq-1)q}{2-q}\|f\|_{L^q(\Omega)}^q\log
  \frac{\|f\|_{L^2(\Omega)}^2}{\|f\|_{L^q(\Omega)}^2} \\
  &\le \frac{(pq-1)q}{2-q}\|f\|_{L^q(\Omega)}^q\log\left(\frac{(2-q)C_P^2}{q}\,
  \frac{\|\na f\|_{L^2(\Omega)}^2}{\|f\|_{L^q(\Omega)}^2} + 1\right) \\
  &\le (pq-1)C_P^2\|\na f\|_{L^2(\Omega)}^2\|f\|_{L^q(\Omega)}^{q-2},
\end{align*}
which shows the desired Beckner inequality.

It remains to prove \eqref{b.wirt}. For this, 
we employ the Poincar\'e-Wirtinger inequality \eqref{b.pwi2}
$$
  C_P^2\|\na f\|_{L^2(\Omega)}^2\ge \|f\|_{L^2(\Omega)}^2 - \|f\|_{L^1(\Omega)}^2
$$
and the interpolation inequality \eqref{b.aux4} in the form
$$
  \|f\|_{L^q(\Omega)}^2 \ge \|f\|_{L^1(\Omega)}^{2/\theta}
  \|f\|_{L^2(\Omega)}^{2(\theta-1)/\theta}, \quad \theta =  \frac{q}{2-q} \le 1,
$$
to obtain
\begin{align*}
  \|f\|_{L^q(\Omega)}^2 &+ \frac{2-q}{q}C_P^2\|\na f\|_{L^2(\Omega)}^2
  - \|f\|_{L^2(\Omega)} \\
  &\ge \|f\|_{L^1(\Omega)}^{2/\theta}\|f\|_{L^2(\Omega)}^{2(\theta-1)/\theta}
  + \left(\frac{2-q}{q}-1\right)\|f\|_{L^2(\Omega)}^2
  - \frac{2-q}{q}\|f\|_{L^1(\Omega)}^2.
\end{align*}
We interpret the right-hand side as a function $G$ of $\|f\|_{L^1(\Omega)}^2$.
Then, setting $A=\|f\|_{L^2(\Omega)}^2$,
\begin{align*}
  G(y) &= y^{1/\theta}A^{1-1/\theta} + \frac{2(1-q)}{q}A - \frac{2-q}{q}y, \\
  G'(y) &= \frac{1}{\theta}y^{1/\theta-1}A^{1-1/\theta} - \frac{2-q}{q}, \\
  G''(y) &= \frac{1}{\theta}\left(\frac{1}{\theta}-1\right)
 y^{1/\theta-2}A^{1-1/\theta} \ge 0,
\end{align*}
Therefore, $G$ is a convex function which satisfies $G(A)=0$ and $G'(A)=0$. 
This implies that $G$ is a nonnegative function on $\R^+$ and in particular,
$G(\|f\|_{L^1(\Omega)}^2) \ge 0$.
This proves \eqref{b.wirt}, completing the proof.
\end{proof}


\subsection{Zeroth-order entropies}\label{sec.cont.E}

Let $u$ be a smooth solution to \eqref{1.pme}-\eqref{1.bc} and
let $u_0\in L^\infty(\Omega)$, $\inf_\Omega u_0\geq0$ in $\Omega$. 
By the maximum principle,
$0\leq\inf_\Omega u_0\le u(t)\le \sup_\Omega u_0$ in $\Omega$ for $t\ge 0$.
First, we prove algebraic decay rates for $E_\alpha[u]$, defined in \eqref{1.E}.

\begin{theorem}[Polynomial decay for $E_\alpha$]\label{thm.E.poly}
Let $\alpha>0$ and $\beta>1$. Let $u$ be a smooth solution to 
\eqref{1.pme}-\eqref{1.bc} and $u_0\in L^\infty(\Omega)$ with $\inf_\Omega u_0\geq0$. Then
$$
  E_\alpha[u(t)] \le \frac{1}{(C_1t + C_2)^{(\alpha+1)/(\beta-1)}}, \quad t\ge 0,
$$
where
$$
  C_1 = \frac{4\alpha\beta(\beta-1)}{(\alpha+1)(\alpha+\beta)^2}
  \left(\frac{\alpha+1}{C_B(p,q)}\right)^{(\alpha+\beta)/(\alpha+1)}, \quad
  C_2 = E[u_0]^{-(\beta-1)/(\alpha+1)},
$$
and $C_B(p,q)>0$ is the constant in the Beckner inequality for
$p=(\alpha+\beta)/2$ and $q=2(\alpha+1)/(\alpha+\beta)$.
\end{theorem}

\begin{proof}
We apply Lemma \ref{lem.beckner} with $p=(\alpha+\beta)/2$ and
$q=2(\alpha+1)/(\alpha+\beta)$. The assumptions on $\alpha$ and $\beta$ guarantee that
$0<q<2$ and $pq>1$. Then, with $f=u^{(\alpha+\beta)/2}$,
$$
  E_\alpha[u] = \frac{1}{\alpha+1}\left(
  \int_{\Omega} u^{\alpha+1}dx - \left(\int_{\Omega}udx\right)^{\alpha+1}\right)
  \le \frac{C_B(p,q)}{\alpha+1}\left(\int_{\Omega}|\na u^{(\alpha+\beta)/2}|^2 dx
  \right)^{(\alpha+1)/(\alpha+\beta)}.
$$
Now, computing the derivative, 
\begin{align*}
  \frac{dE_\alpha}{dt}
  &= -\int_{\Omega}\na u^\alpha\cdot\na u^\beta dx 
  = -\frac{4\alpha\beta}{(\alpha+\beta)^2}\int_\Omega |\na u^{(\alpha+\beta)/2}|^2 dx \\
  &\le -\frac{4\alpha\beta}{(\alpha+\beta)^2}
  \left(\frac{\alpha+1}{C_B(p,q)}\right)^{(\alpha+\beta)/(\alpha+1)}
  E_\alpha[u]^{(\alpha+\beta)/(\alpha+1)}.
\end{align*} 
An integration of this inequality gives the assertion.
\end{proof}

Next, we show exponential decay rates.
\begin{theorem}[Exponential decay for $E_\alpha$]\label{thm.E.exp} \
Let $u$ be a smooth solution to \eqref{1.pme}-\eqref{1.bc}, 
$0<\alpha\le 1$, $\beta>0$, $u_0\in L^\infty(\Omega)$ with $\inf_\Omega u_0\geq0$. 
Then
$$
  E_\alpha[u(t)] \le E_\alpha[u_0] e^{-\Lambda t}, \quad t\ge 0.
$$
The constant $\Lambda$ is given by
$$
  \Lambda = \frac{4\alpha\beta}{C_B(\tfrac12(\alpha+1),2)(\alpha+1)}
  \inf_{\Omega}(u_0^{\beta-1})\geq0,
$$
for $\beta>0$ and
$$
  \Lambda = \frac{4\alpha\beta(\alpha+1)}{C_B'(p,q)(\alpha+\beta)^2}\,\|u_0\|_{L^1(\Omega)}^{\beta-1},
$$
for $\beta>1$. Here, $C_B(\tfrac12(\alpha+1),2)$ and $C_B'(p,q)$ 
are the constants in the Beckner inequalities
\eqref{b.beckner} and \eqref{b.beckner2}, respectively, 
with $p=(\alpha+\beta)/2$ and $q=2(\alpha+1)/(\alpha+\beta)$.
\end{theorem}

\begin{proof}
Let $\beta>0$. We compute
\begin{align*}
  \frac{dE_\alpha}{dt} 
  &= -\frac{4\alpha\beta}{(\alpha+1)^2}
  \int_\Omega u^{\beta-1}|\na u^{(\alpha+1)/2}|^2 dx \\
  &\le -\frac{4\alpha\beta}{(\alpha+1)^2}
  \inf_{\Omega}(u_0^{\beta-1})
  \int_\Omega|\na u^{(\alpha+1)/2}|^2 dx.
\end{align*}
By the Beckner inequality \eqref{b.beckner} with $p=(\alpha+1)/2$, $q=2$,
and $f=u^{(\alpha+1)/2}$, we find that
$$
  \frac{dE_\alpha}{dt} \le -\frac{4\alpha\beta}{C_B(p,2)(\alpha+1)}
  \inf_{\Omega}(u_0^{\beta-1})E_\alpha,
$$
and Gronwall's lemma proves the claim. The restriction $p\le 1$ in
Lemma \ref{lem.beckner} is equivalent to $\alpha\le 1$.

Next, let $\beta>1$. By Lemma \ref{lem.beckner2}, with $p=(\alpha+\beta)/2$, 
$q=2(\alpha+1)/(\alpha+\beta)$, and $f=u^{(\alpha+\beta)/2}$, it follows that
$$
  \|u\|_{L^{\alpha+1}(\Omega)}^{\beta-1}\left(
  \int_{\Omega} u^{\alpha+1}dx - \left(\int_{\Omega}udx\right)^{\alpha+1}\right)
  \le C_B'(p,q)\int_{\Omega}|\na u^{(\alpha+\beta)/2}|^2 dx.
$$
Hence, we can estimate
\begin{align*}
  \frac{dE_\alpha}{dt} = -\frac{4\alpha\beta}{(\alpha+\beta)^2}
  \int_\Omega|\na u^{(\alpha+\beta)/2}|^2 dx
  &\le -\frac{4\alpha\beta(\alpha+1)}{(\alpha+\beta)^2}\,
    \frac{\|u\|_{L^{\alpha+1}(\Omega)}^{\beta-1}}{C_B'(p,q)}E_\alpha[u]\\
  &\le -\frac{4\alpha\beta(\alpha+1)}{(\alpha+\beta)^2}\,
  \frac{\|u_0\|_{L^1(\Omega)}^{\beta-1}}{C_B'(p,q)}E_\alpha[u],
\end{align*}
and Gronwall's lemma gives the conclusion.
Note that in the last step of the inequality we used that $\|u\|_{L^{\alpha+1}(\Omega)}\geq\|u\|_{L^1(\Omega)}=\|u_0\|_{L^1(\Omega)}$.
\end{proof}


\subsection{First-order entropies}\label{sec.cont.F}

In this section, we consider the diffusion equation \eqref{1.pme} on
the torus $\Omega=\T^d$. 
We prove the exponential decay for the first-order entropies \eqref{1.F}.

\begin{theorem}[Exponential decay of $F_\alpha$ in one space dimension]
\label{thm.F.exp}
Let $u$ be a smooth solution to \eqref{1.pme} on the one-dimensional
torus $\Omega=\T$. Let $u_0\in L^\infty(\Omega)$ with $\inf_\Omega u_0\geq0$ and let $\alpha,\beta>0$ satisfy $-2\le\alpha-2\beta<1$. Then
$$
  F_\alpha[u(t)] \le F_\alpha[u_0]e^{-\Lambda t}, \quad 0\le t\le T,
$$
where 
$$
  \Lambda = \frac{2\beta}{C_P^2}
  \inf_{\Omega}(u_0^{\alpha+\beta-\gamma-1})\inf_{\Omega}(u_0^{\gamma-\alpha})\geq0, \quad
  \gamma = \frac23(\alpha+\beta-1),
$$
where $C_P>0$ is the Poincar\'e constant in \eqref{b.pwi}.
\end{theorem}

\begin{proof}
We extend slightly the entropy construction method of \cite{JuMa06}.
The time derivative of the entropy reads as
\begin{align*}
  \frac{dF_\alpha}{dt}
  &= \frac{\alpha}{2}\int_\Omega (u^{\alpha/2})_x(u^{\alpha/2-1}u_t)_x dx
  = -\frac{\alpha}{2}\int_\Omega (u^{\alpha/2})_{xx}u^{\alpha/2-1}(u^\beta)_{xx} dx  \\
  &= -\frac{\alpha^2\beta}{4}\int_\Omega u^{\alpha+\beta-1}
  \left(\left(\frac{\alpha}{2}-1\right)(\beta-1)\xi_G^4
  + \left(\frac{\alpha}{2}+\beta-2\right)\xi_G^2\xi_L + \xi_L^2\right)dx, 
\end{align*}
where we introduced
$$
  \xi_G = \frac{u_x}{u}, \quad \xi_L = \frac{u_{xx}}{u}.
$$
This integral is compared to 
$$
  \int_\Omega u^{\alpha+\beta-\gamma-1}(u^{\gamma/2})_{xx}^2 dx
  = \frac{\gamma^2}{4}\int_\Omega u^{\alpha+\beta-1}\left(
  \left(\frac{\gamma}{2}-1\right)^2\xi_G^4 + (\gamma-2)\xi_G^2\xi_L + \xi_L^2\right)dx,
$$
where, compared to the method of \cite{JuMa06}, 
$\gamma\neq 0$ gives an additional degree of freedom in the calculations.
In the one-dimensional situation, there is only one relevant integration-by-parts rule:
$$
  0 = \int_\Omega(u^{\alpha+\beta-4}u_x^3)_x dx
  = \int_\Omega u^{\alpha+\beta-1}\big((\alpha+\beta-4)\xi_G^4
  + 3\xi_G^2\xi_L\big)dx.
$$
We introduce the polynomials
\begin{align}
  S_0(\xi) &= \left(\frac{\alpha}{2}-1\right)(\beta-1)\xi_G^4
  + \left(\frac{\alpha}{2}+\beta-2\right)\xi_G^2\xi_L + \xi_L^2, \label{f.S0} \\
  D_0(\xi) &= \left(\frac{\gamma}{2}-1\right)^2\xi_G^4 + (\gamma-2)\xi_G^2\xi_L 
  + \xi_L^2, \label{f.D0} \\
  T(\xi) &= (\alpha+\beta-4)\xi_G^4 + 3\xi_G^2\xi_L, \nonumber
\end{align}
where $\xi=(\xi_G,\xi_L)$. We wish to show that there exist numbers $c$, $\gamma\in\R$
($\gamma\neq 0$) and $\kappa>0$ such that
$$
  S(\xi) = S_0(\xi) + cT(\xi) - \kappa D_0(\xi) \ge 0 \quad\mbox{for all }\xi\in\R^2.
$$

The polynomial $S$ can be written as 
$S(\xi)=a_1\xi_G^4+a_2\xi_G^2\xi_L+(1-\kappa)\xi_L^2$, where
\begin{align*}
  a_1 &= -\frac14(\gamma-2)^2\kappa + (\alpha+\beta-4)c + \frac12(\alpha-2)(\beta-1), \\
  a_2 &= -(\gamma-2)\kappa + 3c + \frac12(\alpha+2\beta-4).
\end{align*}
Therefore, the maximal value for $\kappa$ is $\kappa=1$. Let $\kappa=1$.
Then we need to eliminate the mixed term $\xi_G^2\xi_L$. 
The solution of $a_2=0$ is given by $c=-\frac16(\alpha+2\beta-2\gamma)$, which leads to 
$$
  a_1 = -\frac14\left(\gamma-\frac23(\alpha+\beta-1)\right)^2
  - \frac{1}{18}(\alpha-2\beta-1)(\alpha-2\beta+2).
$$
Choosing $\gamma=\frac23(\alpha+\beta-1)$ to maximize $a_1$, 
we find that $a_1\ge 0$ and
hence $S(\xi)\ge 0$ if and only if $-2\le\alpha-2\beta\le 1$.

Using the Poincar\'e inequality \eqref{b.pwi} and the maximum principle, we obtain
\begin{align*}
  \frac{dF_\alpha}{dt} 
  &= -\frac{\alpha^2\beta}{4}\int_\Omega u^{\alpha+\beta-1} S_0(\xi) dx 
  = -\frac{\alpha^2\beta}{4}\int_\Omega u^{\alpha+\beta-1}
  (S_0(\xi) + cT(\xi))dx \\
  &\le -\frac{\alpha^2\beta}{4}\int_\Omega u^{\alpha+\beta-1}D_0(\xi)dx 
  = -\frac{\alpha^2\beta}{\gamma^2}\int_\Omega u^{\alpha+\beta-\gamma-1}
  (u^{\gamma/2})_{xx}^2 dx \\
  &\le -\frac{\alpha^2\beta}{\gamma^2}\inf_{\Omega\times(0,\infty)}
  (u^{\alpha+\beta-\gamma-1})\int_\Omega (u^{\gamma/2})_{xx}^2 dx \\
  &\le -\frac{\alpha^2\beta}{\gamma^2 C_P^2}\inf_{\Omega}
  (u_0^{\alpha+\beta-\gamma-1})\int_\Omega (u^{\gamma/2})_x^2 dx \\
  &\le -\frac{2\beta}{ C_P^2}\inf_{\Omega}
  (u_0^{\alpha+\beta-\gamma-1})\inf_{\Omega}(u_0^{\gamma-\alpha})
  F_\alpha.
\end{align*}
For the last inequality, we use that $(u^{\gamma/2})_x=\frac{\gamma}{\alpha} u^{(\gamma-\alpha)/2} (u^{\alpha/2})_x$, which cancels out the ratio ${\alpha^2}/{\gamma^2}$.
An application of the Gronwall's lemma finishes the proof. 
\end{proof}

We turn to the multi-dimensional case.

\begin{theorem}[Exponential decay of $F_\alpha$ in several space dimensions]\
\label{thm.F.exp3}
Let $u$ be a smooth solution to \eqref{1.pme} on the torus $\Omega=\T^d$. 
Let $u_0\in L^\infty(\Omega)$ with $\inf_\Omega u_0>0$ and let
\begin{align*}
  (\alpha,\beta)\in M_d = \big\{(\alpha,\beta)\in\R^2:&\
  (2-2\alpha+2\beta-d+\alpha d)(4-4\beta-2d+\alpha d+2\beta+2\beta d)>0 \\
  &\mbox{and }(\alpha-2\beta-1)(\alpha-2\beta+2)<0\big\}
\end{align*}
(see Figure \ref{fig.ab}). Then there exists $\Lambda>0$, depending on
$\alpha$, $\beta$, $d$, $u_0$, and $\Omega$ such that
$$
  F_\alpha[u(t)] \le F_\alpha[u_0]e^{-\Lambda t}, \quad t\ge 0.
$$
\end{theorem}

\begin{figure}[htp]
  \includegraphics[width=120mm]{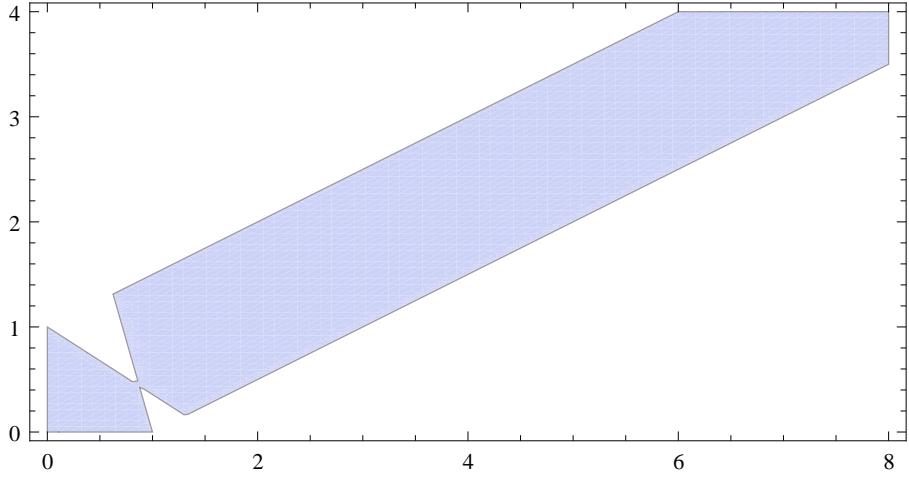}
  \caption{Illustration of the set $M_d$, defined in Theorem \ref{thm.F.exp3},
  for $d=9$.}
  \label{fig.ab}
\end{figure}

\begin{proof}
The time derivative of the first-order entropy becomes
\begin{equation}\label{f.aux}
  \frac{dF_\alpha}{dt} = -\frac{\alpha}{2}\int_\Omega u^{\alpha/2-1}
  \Delta(u^{\alpha/2})\Delta(u^\beta)dx \\
  = -\frac{\alpha^2\beta}{4}\int_\Omega u^{\alpha+\beta-1}S_0 dx,
\end{equation}
where $S_0$ is defined in \eqref{f.S0} with the (scalar) variables
$\xi_G=|\na u|/u$, $\xi_L=\Delta u/u$. We compare this integral to
$$
  \int_\Omega u^{\alpha+\beta-\gamma-1}(\Delta(u^{\gamma/2}))^2 dx
  = \frac{\gamma^2}{4}\int_\Omega u^{\alpha+\beta-1}D_0 dx,
$$
where $D_0$ is as in \eqref{f.D0} and $\gamma\neq 0$. 
In contrast to the one-dimensional case,
we employ two integration-by-parts rules:
\begin{align*}
  0 &= \int_\Omega\diver\big(u^{\alpha+\beta-4}|\na u|^2\na u\big)dx 
  = \int_\Omega u^{\alpha+\beta-1}T_1 dx, \\
  0 &= \int_\Omega\diver\big(u^{\alpha+\beta-3}(\na^2 u-\Delta{\mathbb I})
  \cdot\na u\big) dx = \int_\Omega u^{\alpha+\beta-1}T_2 dx,
\end{align*}
where
\begin{align*}
  T_1 &= (\alpha+\beta-4)\xi_G^4 + 2\xi_{GHG} + \xi_G^2\xi_L, \\
  T_2 &= (\alpha+\beta-3)\xi_{GHG} - (\alpha+\beta-3)\xi_G^2\xi_L
  + \xi_{H}^2 - \xi_L^2,
\end{align*}
and $\xi_{GHG}=u^{-3}\na u^\top\na^2 u\na u$, $\xi_{H}=u^{-1}\|\na^2 u\|$. Here, $\|\na^2 u\|$ denotes the Frobenius norm of the hessian.\\
In order to compare $\na^2 u$ and $\Delta u$, we employ Lemma 2.1 of \cite{JuMa08}:
$$
  \|\na^2 u\|^2 \ge \frac{1}{d}(\Delta u)^2 
  + \frac{d}{d-1}\left(\frac{\na u^\top\na^2 u\na u}{|\na u|^2} - \frac{\Delta u}{d}
  \right)^2.
$$
Therefore, there exists $\xi_R\in\R$ such that
$$
  \xi_{H}^2=\frac{\xi_L^2}{d} + \frac{d}{d-1}
  \left(\frac{\xi_{GHG}}{\xi_G^{2}}-\frac{1}{d}\xi_L\right)^2+\xi_R^2
  = \frac{\xi_L^2}{d} + \frac{d}{d-1}\xi_S^2 + \xi_R^2,
$$
where we introduced $\xi_S=\xi_{GHG}/\xi_G^2-\xi_L/d$.
Rewriting the polynomials $T_1$ and $T_2$ in terms of 
$\xi=(\xi_G,\xi_L,\xi_S,$ $\xi_R)\in\R^4$ leads to:
\begin{align*}
  T_1(\xi) &= (\alpha+\beta-4)\xi_G^4 + \frac{2+d}{d}\xi_G^2\xi_L
  + 2\xi_G^2\xi_S, \\
  T_2(\xi) &= \frac{1-d}{d}(\alpha+\beta-3)\xi_G^2\xi_L+\frac{1-d}{d}\xi_L^2
  +\xi_S\xi_G^2(\alpha+\beta-3)+\frac{d}{d-1}\xi_S^2+\xi_R^2.
\end{align*}
We wish to find $c_1$, $c_2$, $\gamma\in\R$ ($\gamma\neq 0$) and $\kappa>0$ such that
\begin{equation*}
  S(\xi) = S_0(\xi) + c_1T_1(\xi) + c_2T_2(\xi) - \kappa D_0(\xi)\ge 0
  \quad\mbox{for all }\xi\in\R^{4}.
\end{equation*}
The polynomial $S$ can be written as
\begin{align*}
  S(\xi) &= a_1\xi_G^4 + a_2\xi_G^2\xi_L + a_3\xi_L^2 + a_4\xi_G^2\xi_S
  + a_5\xi_S^2 + c_2\xi_R^2, \mbox{ where} \\
  a_1 &= \left(\frac{\alpha}{2}-1\right)(\beta-1) + (\alpha+\beta-4)c_1
  - \left(\frac{\gamma}{2}-1\right)^2\kappa, \\
  a_2 &= \frac{\alpha}{2}+\beta-2 + \left(\frac{2}{d}+1\right)c_1 
  - (\alpha+\beta-3)\frac{d-1}{d}c_2 - (\gamma-2)\kappa, \\
  a_3 &= 1+\frac{1-d}{d}c_2-\kappa, \\
  a_4 &= 2c_1 + (\alpha+\beta-3)c_2, \\
  a_5 &= \frac{d}{d-1}c_2.
\end{align*}
We remove the variable $\xi_R$ by requiring that
$c_2\ge 0$. The remaining polyomial can be reduced to a quadratic
polynomial by setting $x=\xi_L/\xi_G^2$ and $y=\xi_S/\xi_G^2$:
\begin{equation}\label{pos2}
  S(x,y) \ge a_1 + a_2 x + a_3 x^2 + a_4 y + a_5 y^2 \ge 0\quad\mbox{for all }
  x,\,y\in\R.
\end{equation}
This quadratic decision problem can be solved by employing the computer
algebra system {\tt Mathematica}. The result of the command
\begin{quote}
  \begin{verbatim} 
  Resolve[ForAll[{x, y}, Exists[{C1, C2, kappa, gamma}, 
  a1 + a2*x + a3*x^2 + a4*y + a5*y^2 >= 0 && kappa > 0 
  && gamma != 0]], Reals]
  \end{verbatim}
\end{quote}
\vskip-4mm
gives all $(\alpha,\beta)\in\R^2$ such that there exist $c_1$, $c_2$,
$\gamma\in\R$ ($\gamma\neq 0$) and $\kappa>0$ such that \eqref{pos2} holds.
This region equals the set $M_d$, defined in the statement of the theorem.

Similar to the one-dimensional case, we infer that
$$
  \frac{dF_\alpha}{dt} 
  \le -\frac{\alpha^2\beta\kappa}{4}\int_\Omega u^{\alpha+\beta-1}D_0(\xi)dx
  = -\frac{\alpha^2\beta\kappa}{\gamma^2}\int_\Omega u^{\alpha+\beta-\gamma-1}
  (\Delta u^{\gamma/2})^2 dx. 
$$
Thus, proceeding as in the proof of Theorem \ref{thm.F.exp} and using
the identity
$$
  \int_\Omega (\Delta f)^2 dx = \int_\Omega \|\na^2 f\|^2 dx
$$
for smooth functions $f$ (which can be obtained by integration by parts twice), we obtain
$$
  \frac{dF_\alpha}{dt}
  \le -\frac{2\beta\kappa}{C_P^2}
  \inf_\Omega(u_0^{\alpha+\beta-\gamma-1})
  \inf_\Omega(u_0^{\gamma-\alpha})F_\alpha.
$$
Gronwall's lemma concludes the proof.
\end{proof}

\begin{remark}\label{rem.F}\rm
Under the additional constraints $a_2=a_3=0$, we are able to solve the
decision problem \eqref{pos2} without the help of the computer algebra system.
The solution set, however, is slightly smaller than $M_d$ which is obtained from
{\tt Mathematica} without these constraints. Indeed, we can compute
$c_1$ and $c_2$ from the equations $a_2=a_3=0$ giving
$$
  c_1 =  \frac{d}{d+2}\left(\frac{\alpha}{2} - 1 + \kappa(1+\gamma-\alpha-\beta)\right),
  \quad c_2 = \frac{d(1 - \kappa)}{d-1}.
$$
The decision problem \eqref{pos2} reduces to
$$
  a_1 + a_4 y + a_5 y^2 \ge 0\quad\mbox{for all }y\in\R.
$$
If $\kappa<1$, it holds $c_2>0$ and consequently, $a_5>0$. Therefore,
the above polynomial is nonnegative for all $y\in\R$ if it has no real roots, 
i.e., if 
$$
  0 \le 4a_1a_5-a_4^2 = q_0 + q_1\gamma + q_2\gamma^2
$$
for some $\gamma\neq 0$, where (for $d>1$)
$$
  q_2 = -\frac{d^2\kappa}{(d+2)^2(d-1)^2}\big(3d(d-4)\kappa + (d+2)^2\big) < 0,
$$
and $q_0$, $q_1$ are polynomials depending on $d$, $\alpha$, $\beta$, and $\kappa$.
The above problem is solvable if and only if there exist real roots, i.e.\ if
$$
  0 \le q_1^2-4q_0q_2 
  = \frac{4\kappa(1-\kappa)}{(d+2)^2(d-1)^2}(s_0+s_1\kappa+s_2\kappa^2),
$$
where
\begin{align*}
  s_0 &= -d(5d-8) + 6d(d-1)\alpha + 2d(d+2)\beta + 2(d+2)\alpha\beta
  - (2d^2+1)\alpha^2 - (d+2)^2\beta^2, \\
  s_1 &= 2d(3d-4) - 2d(4d-3)\alpha - 4d(d+1)\beta + 2d(3d-5)\alpha\beta
  + 2d(d + 1)\alpha^2 \\
  &\phantom{xx}{}- 2d(d-6)\beta^2, \\
  s_2 &= -d^2(\alpha+\beta-1)^2.
\end{align*}
We set $f(\kappa)=s_0+s_1\kappa+s_2\kappa^2$. We have to find $0<\kappa<1$ such
that $f(\kappa)\ge 0$. Since $s_2\le 0$, this is possible if $f(\kappa)$
possesses two (not necessarily distinct) real roots $\kappa_0$ and $\kappa_1$
and if at least one of these roots is between zero and one. 
Since $f(1)=-(d-1)^2(\alpha-2\beta)^2\le 0$, there are only two
possibilities for $\kappa_0$ and $\kappa_1$: either
$\kappa_0\le 0\le \kappa_1\le 1$ or $0\le\kappa_0\le\kappa_1\le 1$.
The first case holds if $f(0)=s_0\ge 0$, the second one if
\begin{align}
  f'(0) &= s_1 \ge 0, \quad f'(1) = s_1+2s_2 \le 0, \label{f1} \\
  s_1^2-4s_0s_2 &= -4d^2(\alpha-2\beta+2)(\alpha-2\beta-1)(4-2d+d\alpha+2d\beta) 
  \label{f2} \\
  &\phantom{xx}{}\times(2-d+(d-2)\alpha+2\beta) \ge 0. \nonumber
\end{align}
The set of all $(\alpha,\beta)\in\R^2$ fulfilling these conditions
is illustrated in Figure \ref{fig.ab2}. 
\qed
\end{remark}

\begin{figure}[htp]
  \includegraphics[width=120mm]{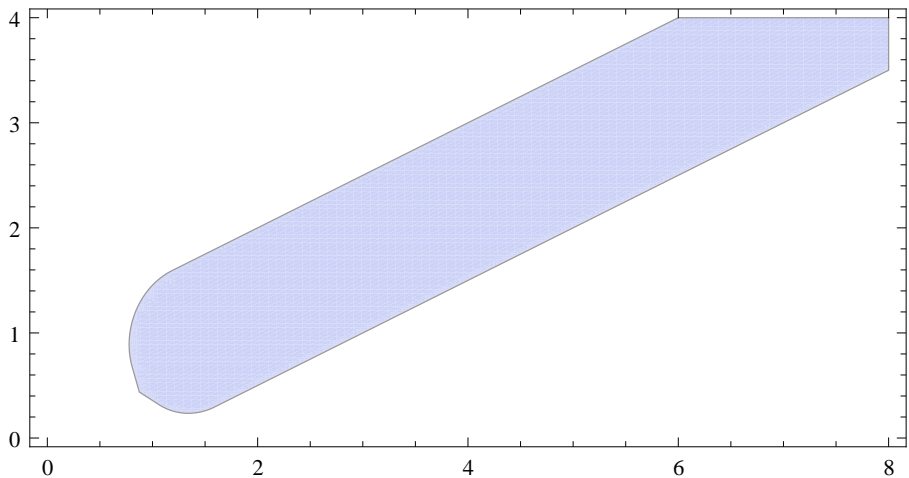}
  \caption{Set of all $(\alpha,\beta)$ fulfilling $s_0\ge 0$, \eqref{f1}, 
  and \eqref{f2} for $d=9$.}
  \label{fig.ab2}
\end{figure}


\section{The discrete case}\label{sec.disc}

We introduce the finite-volume scheme and prove discrete versions
of the generalized Beckner inequality as well as the discrete decay rates.

\subsection{Notations and finite-volume scheme}\label{sec.not}

Let $\Omega$ be an open bounded polyhedral subset of $\R^d$ ($d\ge 2$) with
Lipschitz boundary and $\mbox{m}(\Omega)=1$. 
An admissible mesh of $\Omega$ is given by a family
$\TT$ of control volumes (open and convex polyhedral subsets of $\Omega$ with
positive measure);
a family $\E$ of relatively open parts of hyperplanes in $\R^d$ which represent
the faces of the control volumes; 
and a family of points $(x_K)_{K\in\TT}$ which satisfy Definition 9.1 in
\cite{EGH00}. This definition implies that the straight line between two
neighboring centers of cells $(x_K,x_L)$ is orthogonal to the edge $\sigma=K|L$
between the two control volume $K$ and $L$.
For instance, triangular meshes in $\R^2$ satisfy the admissibility condition 
if all angles of the triangles are smaller than $\pi/2$ \cite[Examples 9.1]{EGH00}.
Voronoi meshes in $\R^d$ are also admissible meshes \cite[Examples 9.2]{EGH00}.

We distinguish the interior faces $\sigma\in\E_{\rm int}$ and the boundary faces
$\sigma\in\E_{\rm ext}$. Then the union $\E_{\rm int}\cup\E_{\rm ext}$ equals
the set of all faces $\E$.
For a control volume $K\in\TT$, we denote by $\E_K$ the set
of its faces, by $\E_{{\rm int},K}$ the set of its interior faces, and by
$\E_{{\rm ext},K}$ the set of edges of $K$ included in $\pa\Omega$.

Furthermore, we denote by $\dd$ the distance in $\R^d$. We assume that the family of meshes satisfies the
following regularity requirement: There exists $\xi>0$ such that for all
$K\in\TT$ and all $\sigma\in\E_{{\rm int},K}$ with $\sigma=K|L$, it holds
\begin{equation}\label{2.ass}
  \dd(x_K,\sigma)\ge \xi \dd(x_K,x_L).
\end{equation}
This hypothesis is needed to apply a discrete Poincar\'e inequality; see Lemma
\ref{lem.dpwi}. Introducing for $\sigma\in\E$ the notation
$$
  d_\sigma = \left\{\begin{array}{ll}
  \dd(x_K,x_L) &\quad\mbox{if }\sigma\in\E_{\rm int},\ \sigma=K|L, \\
  \dd(x_K,\sigma) &\quad\mbox{if }\sigma\in\E_{{\rm ext},K},
  \end{array}\right.
$$
we define the transmissibility coefficient 
$$
  \tau_\sigma = \frac{\m(\sigma)}{d_\sigma}, \quad \sigma\in\E.
$$
The size of the mesh is denoted by $\triangle x = \max_{K\in\TT}\text{diam}(K)$.
Let $T>0$ be some final time and $M_T$ the number of time steps. Then the time step size and the time points are given by, respectively,
$\triangle t = \frac{T}{M_T}$, $t^k = k\triangle t$ for $0\le k\le M_T$.
We denote by $\D$ an admissible space-time discretization of $\Omega_{T}=\Omega \times (0,T)$ composed of an admissible mesh $\TT$ of $\Omega$ and the values $\triangle t$ and $M_{T}$.\\
A classical finite volume scheme provides an approximate solution which is constant on each cell of the mesh and on each time interval. Let $X(\TT)$ be the linear space of functions $\Omega\to\R$ which are constant on each cell $K\in\TT$. We define on $X(\TT)$ the discrete $L^p$ norm, discrete $W^{1,p}$ seminorm, and discrete $W^{1,p}$ norm by, respectively,
\begin{align*}
  \|u\|_{0,p,\TT} 
  &= \left(\int_\Omega|u|^p dx\right)^{1/p} 
  = \left(\sum_{K\in\TT}\m(K)|u_K|^p\right)^{1/p}, \\ 
  |u|_{1,p,\TT} &= \bigg(\sum_{\substack{\sigma\in\E_{\rm int},\\ \sigma=K|L}}
  \frac{\m(\sigma)}{d_\sigma^{p-1}}|u_K-u_L|^p\bigg)^{1/p}, \\
  \|u\|_{1,p,\TT} &= \|u\|_{0,p,\TT} + |u|_{1,p,\TT},
\end{align*}
where $u\in X(\TT)$, $u=u_K$ in $K\in\TT$, and $1\le p<\infty$. 
The discrete entropies for $u\in X(\TT)$
are defined analogously to the continuous case:
\begin{align}
  E_\alpha^d[u] &= \frac {1}{\alpha+1}\left(\sum_{K\in\TT}\m(K)u^{\alpha+1}_K
  - \left(\sum_{K\in\TT}\m(K)u_K\right)^{\alpha+1}\right), \label{2.Ed} \\
  F_\alpha^d[u] &= \frac12|u^{\alpha/2}|_{1,2,\TT}^2. \label{2.Fd}
\end{align}

We are now in the position to define the finite-volume scheme of 
\eqref{1.pme}-\eqref{1.bc}. Let $\D$ be a finite-volume discretization
of $\Omega_T$. The initial datum is approximated by its $L^2$ projection on 
control volumes:
\begin{equation}\label{2.n0}
  u^0 = \sum_{K\in\T}u_K^0\mathbf{1}_K, \quad\mbox{where }
  u^0_K = \frac{1}{\m(K)}\int_K u_0(x)dx,
\end{equation}
and $\mathbf{1}_K$ is the characteristic function on $K$. 
Then $\|u^0\|_{0,1,\TT}=\|u_0\|_{L^1(\Omega)}$.
The numerical scheme reads as follows:
\begin{equation}\label{2.fv}
  \m(K)\frac{u_K^{k+1} - u_K^{k}}{\triangle t}
  + \sum_{\substack{\sigma\in\E_{\rm int},\\ \sigma=K|L}}\tau_\sigma
  \big((u_K^{k+1})^\beta-(u_L^{k+1})^\beta\big) = 0,
\end{equation}
for all $K\in\TT$ and $k=0,\ldots,M_T-1$. 
This scheme is based on a fully implicit Euler discretization in time and a 
finite-volume approach for the volume variable. The Neumann boundary conditions \eqref{1.bc} are taken into account as the sum in \eqref{2.fv} applies only on 
the interior edges. The implicit scheme allows us to establish 
discrete entropy-dissipation estimates which would not be possible with an explicit 
scheme.

We summarize in the next proposition the classical results of existence, uniqueness and stability of the solution to the finite volume scheme \eqref{2.n0}-\eqref{2.fv}.

\begin{proposition}
Let $u_0\in L^{\infty}(\Omega)$, $m\geq 0$, $M\geq 0$ such that $m\leq u_0\leq M$ 
in $\Omega$.  Let $\TT$ be an admissible mesh of $\Omega$. Then the scheme \eqref{2.n0}-\eqref{2.fv} admits a unique solution $(u_K^k)_{K\in\TT,\, 0\leq k\leq M_T}$ satisfying
\begin{align*}
  m\leq u_K^k\leq M,&\quad \mbox{for all }\;K\in\TT,\, 0\leq k\leq M_T,  \\
  \sum_{K\in\TT}{\rm m}(K) u_K^k=\Vert u_0\Vert_{L^1(\Omega)},&\quad\mbox{for all }\; 
  0\leq k\leq M_T. 
\end{align*}
\end{proposition}

We refer, for instance, to \cite{EGH00} and \cite{EGHM02} for the proof of this proposition.


\subsection{Discrete generalized Beckner inequalities}\label{sec.disc.bec}

The decay properties rely on discrete generalized Beckner inequalities which
follow from the discrete Poincar\'e-Wirtinger inequality \cite[Theorem~5]{BCF12}:

\begin{lemma}[Discrete Poincar\'e-Wirtinger inequality]\label{lem.dpwi}
Let $\Omega\subset\R^d$ be an open bounded polyhedral set and let ${\mathcal T}$ be an
admissible mesh satisfying the regularity constraint \eqref{2.ass}.
Then there exists a constant $C_p>0$ only
depending on $d$ and $\Omega$ such that for all $f\in X({\mathcal T})$,
\begin{equation}\label{d.poincare}
  \|f - \bar f\|_{0,2,\TT} \le \frac{C_p}{\xi^{1/2}}|f|_{1,2,\TT}
\end{equation}
where $\bar f=\int_\Omega fdx$ (recall that $\mbox{m}(\Omega)=1$)
and $\xi$ is defined in \eqref{2.ass}.
\end{lemma}

\begin{lemma}[Discrete generalized Beckner inequality I]\label{lem.d.bec}
Let $0<q<2$, $pq>1$ or $q=2$ and $0<p\le 1$, 
and $f\in X(\TT)$. Then
$$
  \int_\Omega |f|^qdx - \left(\int_\Omega |f|^{1/p}dx\right)^{pq}
  \le C_b(p,q)|f|_{1,2,\TT}^q
$$
holds, where
$$
  C_b(p,q) = \frac{2(pq-1)C_p^q}{(2-q)\xi^{q/2}}\quad\mbox{if }q<2, \quad
  C_b(p,2) = \frac{C_p^2}{\xi}\quad\mbox{if }q=2.
$$
$C_p$ is the constant in the discrete Poincar\'e-Wirtinger inequality,
and $\xi$ is defined in \eqref{2.ass}.
\end{lemma}

\begin{proof}
The proof follows the lines of the proof of Lemma \ref{lem.beckner} noting that
in the discrete (finite-dimensional) setting, we do not need anymore 
the lower bound on $p$.
Indeed, if $q=2$, the conclusion results from the discrete 
Poincar\'e-Wirtinger inequality \eqref{d.poincare} and the Jensen inequality.
If $q<2$, let $f\in X(\TT)$. Then we have from \eqref{b.aux3} and \eqref{b.aux}
\begin{align}
  \int_\Omega|f|^q dx - \left(\int_\Omega |f|^{1/p}dx\right)^{pq}
  &\le (pq-1)\int_\Omega|f|^q\log\frac{|f|^q}{\|f\|_{0,q,\TT}^q}dx \nonumber \\
  &\le \frac{2(pq-1)}{2-q}\|f\|_{0,q,\TT}^q\log
  \frac{\|f\|_{0,2,\TT}^q}{\|f\|_{0,q,\TT}^q}. \label{d.aux}
\end{align}
We employ the discrete Poincar\'e-Wirtinger inequality \eqref{d.poincare},
$$
  \|f\|_{0,2,\TT}^2 - \|f\|_{0,1,\TT}^2 = \|f-\bar f\|_{0,2,\TT}^2
  \le C_p^2\xi^{-1}|f|_{1,2,\TT}^2,
$$
which implies, as in the proof of Lemma \ref{lem.gpw} (see \eqref{b.aux2}), that
$$
  \|f\|_{0,2,\TT}^q \le C_p^q\xi^{-q/2}|f|_{1,2,\TT}^q + \|f\|_{0,q,\TT}^q.
$$
After inserting this inequality into ´\eqref{d.aux}
to replace $\|f\|_{0,2,\TT}$ and using $\log(x+1)\le x$ for $x\ge 0$, 
the lemma follows.
\end{proof}

The following result is proved exactly as in Lemma \ref{lem.beckner2}
using the discrete Poincar\'e-Wirtinger inequality \eqref{d.poincare}
instead of \eqref{b.pwi}.

\begin{lemma}[Discrete generalized Beckner inequality II]\label{lem.d.bec2}
Let $0<q<2$, $pq\ge 1$, and $f\in X(\TT)$. Then
$$
 \|f\|_{0,q,\TT}^{2-q}\left(\int_\Omega |f|^qdx 
 - \left(\int_\Omega |f|^{1/p}dx\right)^{pq}\right)
  \le C'_b(p,q)|f|_{1,2,\TT}^2
$$
holds, where
$$
  C'_b(p,q) = \left\{\begin{array}{ll}
  \dfrac{q(pq-1)C_p^2}{(2-q)\xi} &\mbox{ if }1\le q<2, \\[3mm]
  \dfrac{(pq-1)C_p^2}{\xi} &\mbox{ if }0<q<1,
  \end{array}\right.
$$
$C_p$ is the constant in the discrete Poincar\'e-Wirtinger inequality,
and $\xi$ is defined in \eqref{2.ass}.
\end{lemma}


\subsection{Zeroth-order entropies}\label{sec.disc.E}

We prove a result which is the discrete analogue of Theorem \ref{thm.E.poly}.
Recall that the discrete entropies $E_\alpha^d[u^k]$ are defined in \eqref{2.Ed}.

\begin{theorem}[Polynomial decay of $E^d_\alpha$]\label{thm.E.poly2}
Let $\alpha>0$ and $\beta>1$. Let $(u_K^k)_{K\in\TT,k\ge 0}$ be
a solution to the finite-volume scheme \eqref{2.fv} with $\inf_{K\in\TT}u_K^0\geq0$. Then
$$
  E^d_\alpha[u^k] \le \frac{1}{(c_1 t^k+c_2)^{(\alpha+1)/(\beta-1)}}, \quad
  k\ge 0,
$$
where
\begin{align*}
  c_1 &= (\beta-1)\left(\frac{(\alpha+1)(\alpha+\beta)^2}{4\alpha\beta}
  \left(\frac{C_b(p,q)}{\alpha+1}\right)^{(\alpha+\beta)/(\alpha+1)}
  +(\alpha+\beta)\triangle t E^d_\alpha[u^0]^{(\alpha+1)/(\beta-1)}\right)^{-1}, \\
  c_2 &= E^d_\alpha[u^0]^{-(\beta-1)/(\alpha+1)},
\end{align*}
and $C_b(p,q)$ for $p=(\alpha+\beta)/2$ and $q=2(\alpha+1)/(\alpha+\beta)$
is defined in Lemma \ref{lem.d.bec}.
\end{theorem}

\begin{proof}
The idea is to ``translate'' the proof of Theorem \ref{thm.E.poly} to the discrete
case. To this end, we use the elementary inequality
$y^{\alpha+1}-x^{\alpha+1}\le (\alpha+1)y^{\alpha}(y-x)$, which follows from
the convexity of the mapping $x\mapsto x^{\alpha+1}$, and scheme \eqref{2.fv}:
\begin{align*}
  E^d_\alpha[u^{k+1}] - E^d_\alpha[u^k]
  &= \frac{1}{\alpha+1}\sum_{K\in\TT}\m(K)
  \big((u_K^{k+1})^{\alpha+1} - (u_K^{k})^{\alpha+1}\big) \\
  &\le \sum_{K\in\TT}\m(K)(u_K^{k+1})^\alpha(u_K^{k+1}-u_K^k) \\
  &\le -\triangle t\sum_{K\in\TT}\sum_{\substack{\sigma\in\E_{\rm int},\\ 
  \sigma=K|L}}\tau_\sigma
  (u_K^{k+1})^\alpha\big((u_K^{k+1})^\beta - (u_L^{k+1})^\beta\big).
\end{align*}
Rearranging the sum leads to
\begin{equation}\label{d.aux2}
  E^d_\alpha[u^{k+1}] - E^d_\alpha[u^k]
  \le -\triangle t\sum_{\substack{\sigma\in\E_{\rm int},\\ \sigma=K|L}}\tau_\sigma
  \big((u_K^{k+1})^\alpha-(u_L^{k+1})^\alpha\big)
  \big((u_K^{k+1})^\beta - (u_L^{k+1})^\beta\big).
\end{equation}
Then, employing the inequality in Lemma \ref{lem.ineq1} (see the appendix), 
it follows that
\begin{align*}
  E^d_\alpha[u^{k+1}] - E^d_\alpha[u^k]
  &\le -\frac{4\alpha\beta\triangle t}{(\alpha+\beta)^2}
  \sum_{\substack{\sigma\in\E_{\rm int},\\ \sigma=K|L}}\tau_\sigma
  \big((u_K^{k+1})^{(\alpha+\beta)/2} - (u_L^{k+1})^{(\alpha+\beta)/2}\big)^2 \\
  &\le -\frac{4\alpha\beta\triangle t}{(\alpha+\beta)^2}
  |(u^{k+1})^{(\alpha+\beta)/2}|_{1,2,\TT}^2,
\end{align*}
and applying Lemma \ref{lem.d.bec} with $p=(\alpha+\beta)/2$,
$q=2(\alpha+1)/(\alpha+\beta)$, and $f=(u^{k+1})^{(\alpha+\beta)/2}$,
$$
  E^d_\alpha[u^{k+1}] - E^d_\alpha[u^k]
  \le -\frac{4\alpha\beta\triangle t}{(\alpha+\beta)^2}
  \left(\frac{\alpha+1}{C_b(p,q)}\right)^{(\alpha+\beta)/(\alpha+1)}
  E^d_\alpha[u^{k+1}]^{(\alpha+\beta)/(\alpha+1)}.
$$
The discrete nonlinear Gronwall lemma (see Corollary \ref{coro.gron2} in the appendix) with 
$$
  \tau = \frac{4\alpha\beta\triangle t}{(\alpha+\beta)^2}
  \left(\frac{\alpha+1}{C_b(p,q)}\right)^{(\alpha+\beta)/(\alpha+1)}, \quad
  \gamma = \frac{\alpha+\beta}{\alpha+1} > 1,
$$
implies that
$$
  E^d_\alpha[u^k] \le \frac{1}{(E^d_\alpha[u^0]^{1-\gamma} + c_1 t^k)^{1/(\gamma-1)}},
  \quad k\ge 0,
$$
where $c_1=(\gamma-1)/(1+\gamma\tau E^d_\alpha[u^0]^{\gamma-1})$.
Finally, computing $c_1$ shows the result.
\end{proof}

The discrete analogue to Theorem \ref{thm.E.exp} is as follows.

\begin{theorem}[Exponential decay for $E^d_\alpha$]\label{thm.E.exp2}
Let $(u_K^k)_{K\in\TT,k\ge 0}$ be a solution to the 
finite-volume scheme \eqref{2.fv} and let $0<\alpha\le 1$, $\beta>0$, $\inf_{K\in\TT}u_K^0\geq0$.
Then
$$
  E^d_\alpha[u^k] \le E^d_\alpha[u^0]e^{-\lambda t^k}, \quad k\ge 0.
$$
The constant $\lambda$ is given by
$$
  \lambda = \frac{4\alpha\beta}{C_b(\tfrac12(\alpha+1),2)(\alpha+1)}
  \inf_{K\in\TT}\big((u_K^0)^{\beta-1}\big)\geq0,
$$
for $\beta>0$, and
$$
  \lambda = \frac{4\alpha\beta(\alpha+1)}{C'_b(p,q)(\alpha+\beta)^2}
  \,\|u^0\|_{0,1,\TT}^{\beta-1}
$$
for $\beta>1$. Here $C'_b(p,q)>0$ is the constant from Lemma \ref{lem.d.bec2} with
$p=(\alpha+\beta)/2$ and $q=2(\alpha+1)/(\alpha+\beta)$.
\end{theorem}

\begin{proof}
Let $\alpha\le 1$ and $\beta>0$.
As in the proof of Theorem \ref{thm.E.poly2}, we find that (see \eqref{d.aux2})
$$
  E^d_\alpha[u^{k+1}] - E^d_\alpha[u^k]
  \le -\triangle t\sum_{\substack{\sigma\in\E_{\rm int},\\ \sigma=K|L}}\tau_\sigma
  \big((u_K^{k+1})^\alpha-(u_L^{k+1})^\alpha\big)
  \big((u_K^{k+1})^\beta - (u_L^{k+1})^\beta\big).
$$
Employing Corollary \ref{coro.ineq1} (see the appendix), we obtain
\begin{align*}
  E^d_\alpha[u^{k+1}] - E^d_\alpha[u^k]
  &\le -\frac{4\alpha\beta\triangle t}{(\alpha+1)^2}
  \sum_{\substack{\sigma\in\E_{\rm int},\\ \sigma=K|L}}\tau_\sigma
  \min\big\{(u_K^{k+1})^{\beta-1},(u_L^{k+1})^{\beta-1}\big\} \\
  &\phantom{xx}{}\times
  \big((u_K^{k+1})^{(\alpha+1)/2} - (u_L^{k+1})^{(\alpha+1)/2}\big)^2 \\
  &\le -\frac{4\alpha\beta\triangle t}{(\alpha+1)^2}
  \inf_{K\in\TT}(u_K^{k+1})^{\beta-1}
  |(u^{k+1})^{(\alpha+1)/2}|_{1,2,\TT}^2 \\
  &\le -\frac{4\alpha\beta\triangle t}{C_b(\tfrac12(\alpha+1),2)(\alpha+1)}
  \inf_{K\in\TT}(u_K^{0})^{\beta-1} E^d_\alpha[u^{k+1}],
\end{align*}
where we have used Lemma \ref{lem.d.bec} with $p=(\alpha+1)/2$, $q=2$, and
$f=u^{(\alpha+1)/2}$. Now, the Gronwall lemma shows the claim.

Next, let $\beta>1$. As in the proof of Theorem \ref{thm.E.poly2}, we find that 
$$
  E^d_\alpha[u^{k+1}] - E^d_\alpha[u^k]
  \le -\frac{4\alpha\beta\triangle t}{(\alpha+\beta)^2}
  |(u^{k+1})^{(\alpha+1)/2}|_{1,2,\TT}^2.
$$
We apply Lemma \ref{lem.d.bec2} with $p=(\alpha+\beta)/2$, 
$q=2(\alpha+1)/(\alpha+\beta)$, and $f=u^{(\alpha+\beta)/2}$ to obtain
\begin{align*}
  E^d_\alpha[u^{k+1}] - E^d_\alpha[u^k]
  &\le -\frac{4\alpha\beta(\alpha+1)\triangle t}{(\alpha+\beta)^2}
  \frac{\|u^{k+1}\|_{0,\alpha+1,\TT}^{\beta-1}}{C'_b(p,q)}
  E^d_\alpha[u^{k+1}]\\
  &\le -\frac{4\alpha\beta(\alpha+1)\triangle t}{(\alpha+\beta)^2}
  \frac{\|u^0\|_{0,1,\TT}^{\beta-1}}{C'_b(p,q)}
  E^d_\alpha[u^{k+1}].
\end{align*}
Then Gronwall's lemma finishes the proof.
\end{proof}


\subsection{First-order entropies}\label{sec.disc.F}

We consider the diffusion equation \eqref{1.pme} on the half open unit cube $[0,1)^d\subset\R^d$ with multiperiodic boundary conditions (this is topologically equivalent to the torus $\T^d$). By identifying
``opposite'' faces on $\pa\Omega$, we can construct a family of control
volumes and a family of edges in such a way that every face 
is an interior face. Then cells with such identified faces are neighboring cells.

\begin{theorem}[Exponential decay of $F^d_\alpha$]\label{thm.F.exp2}
Let $(u_K^k)_{K\in\TT,\,k\ge 0}$ 
be a solution to the finite-volume scheme \eqref{2.fv} 
with $\Omega=\T^d$ and $\inf_{K\in\TT}u_K^0\geq0$. Then, for all $1\le\alpha\le 2$ and $\beta=\alpha/2$,
$$
  F^d_\alpha[u^{k+1}] \le F^d_\alpha[u^{k}], \quad k\in\N.
$$
Furthermore, if $d=1$ and the grid is uniform with $N$ subintervals,
$$
  F^d_\alpha[u^k] \le F^d_\alpha[u_0] e^{-\lambda t^k},
$$
where $\lambda=4\beta\sin^2(\pi/N)\min_i((u_i^0)^{2(\beta-1)})\geq0$.
\end{theorem}


\begin{proof}
The difference $G_\alpha = F^d_\alpha[u^{k+1}] - F^d_\alpha[u^k]$
can be written as
$$
  G_\alpha = \frac12\sum_{\substack{\sigma\in\E_{\rm int},\\ \sigma=K|L}}
  \tau_\sigma\left(
  \big((u_K^{k+1})^{\alpha/2}-(u_L^{k+1})^{\alpha/2}\big)^2
  - \big((u_K^{k})^{\alpha/2}-(u_L^{k})^{\alpha/2}\big)^2\right).
$$
Introducing $a_K = (u_K^{k+1}-u_K^k)/\tau$, we find that
$$
  G_\alpha = \frac12\sum_{\substack{\sigma\in\E_{\rm int},\\ \sigma=K|L}}
  \tau_\sigma\left(
  \big((u_K^{k+1})^{\alpha/2}-(u_L^{k+1})^{\alpha/2}\big)^2
  - \big((u_K^{k+1}-\tau a_K)^{\alpha/2}-(u_L^{k+1}-\tau a_L)^{\alpha/2}\big)^2\right).
$$
We claim that $G_\alpha$ is concave with respect to $\tau$. Indeed, we compute
\begin{align*}
  \frac{\pa G_\alpha}{\pa\tau}
  &= \frac{\alpha}{2}\sum_{\substack{\sigma\in\E_{\rm int},\\ \sigma=K|L}}
  \tau_\sigma
  \big((u_K^{k+1}-\tau a_K)^{\alpha/2}-(u_L^{k+1}-\tau a_L)^{\alpha/2}\big) \\
  &\phantom{xx}{}\times
  \big((u_K^{k+1}-\tau a_K)^{\alpha/2-1}a_K - (u_L^{k+1}-\tau a_L)^{\alpha/2-1}a_L\big),\\
  \frac{\pa^2 G_\alpha}{\pa\tau^2}
  &= -\frac{\alpha^2}{4}\sum_{\substack{\sigma\in\E_{\rm int},\\ \sigma=K|L}}
  \tau_\sigma
  \big((u_K^{k+1}-\tau a_K)^{\alpha/2-1}a_K-(u_L^{k+1}-\tau a_L)^{\alpha/2-1}a_L
  \big)^2 \\
  &\phantom{xx}{}
  - \frac{\alpha}{2}\left(\frac{\alpha}{2}-1\right)
  \sum_{\substack{\sigma\in\E_{\rm int},\\ \sigma=K|L}}\tau_\sigma
  \big((u_K^{k+1}-\tau a_K)^{\alpha/2}-(u_L^{k+1}-\tau a_L)^{\alpha/2}\big) \\
  &\phantom{xx}{}\times
  \big((u_K^{k+1}-\tau a_K)^{\alpha/2-2}a_K^2 
  - (u_L^{k+1}-\tau a_L)^{\alpha/2-2}a_L^2\big).
\end{align*}
Replacing $u_K^{k+1}-\tau a_K$, $u_L^{k+1}-\tau a_L$ by $u_K^{k}$, $u_L^{k}$,
respectively, the second derivative becomes
\begin{align*}
  \frac{\pa^2 G_\alpha}{\pa\tau^2}
  &= - \frac{\alpha^2}{4}
  \sum_{\substack{\sigma\in\E_{\rm int},\\ \sigma=K|L}}\tau_\sigma
  \big((u_K^k)^{\alpha/2-1}a_K - (u_L^k)^{\alpha/2-1}a_L\big)^2 \\
  &\phantom{xx}{}-\frac{\alpha}{2}\left(\frac{\alpha}{2}-1\right)
  \sum_{\substack{\sigma\in\E_{\rm int},\\ \sigma=K|L}}\tau_\sigma
  \big((u_K^k)^{\alpha/2}-(u_L^k)^{\alpha/2}\big)
  \big((u_K^k)^{\alpha/2-2}a_K^2-(u_L^k)^{\alpha/2-2}a_L^2\big) \\
  &= -\frac{\alpha}{4}
  \sum_{\substack{\sigma\in\E_{\rm int},\\ \sigma=K|L}}\tau_\sigma
  (c_1 a_K^2 + c_2 a_K a_L + c_3 a_L^2´),
\end{align*}
where
\begin{align*}
  c_1 &= (\alpha-2)\big((u_K^k)^{\alpha/2}-(u_L^k)^{\alpha/2}\big)(u_K^k)^{\alpha/2-2}
  + \alpha (u_K^k)^{\alpha-2}, \\
  c_2 &= -2\alpha (u_K^k)^{\alpha/2-1}(u_L^k)^{\alpha/2-1}, \\
  c_3 &= -(\alpha-2)\big((u_K^k)^{\alpha/2}-(u_L^k)^{\alpha/2}\big)(u_L^k)^{\alpha/2-2}
  + \alpha (u_L^k)^{\alpha-2}.
\end{align*}
We show that the quadratic polynomial in the variables $a_K$ and $a_L$
is nonnegative for all $u_K^k$ and $u_L^k$. This is the case if and only if
$c_1\ge 0$ and $4c_1c_3-c_2^2\ge 0$. The former condition is equivalent to
$$
  2(\alpha-1)(u_K^k)^{\alpha-2} \ge (\alpha-2)(u_K^k)^{\alpha/2-2}(u_L^k)^{\alpha/2},
$$
which is true for $1\le\alpha\le 2$. After an elementary computation, 
the latter condition becomes
$$
  4c_1c_3-c_2^2
  = 8(\alpha-1)(2-\alpha)(u_K^k)^{\alpha/2-2}(u_L^k)^{\alpha/2-2}
  \big((u_K^k)^{\alpha/2} - (u_L^k)^{\alpha/2}\big)^2 \ge 0
$$
for $1\le\alpha\le 2$. This proves the concavity of $\tau\mapsto G_\alpha(\tau)$.

A Taylor expansion and $G_\alpha(0)=0$ leads to
\begin{align*}
  G_\alpha(\tau) &\le G_\alpha(0) + \tau\frac{\pa G_\alpha}{\pa\tau}(0) \\
  &= \frac{\alpha\tau}{4}\sum_{\substack{\sigma\in\E_{\rm int},\\ \sigma=K|L}}
 \tau_\sigma
  \big((u_K^{k+1})^{\alpha/2}-(u_L^{k+1})^{\alpha/2}\big)
  \big((u_K^{k+1})^{\alpha/2-1}a_K-(u_L^{k+1})^{\alpha/2-1}a_L\big) \\
  &= \frac{\alpha\tau}{4}
  \sum_{\substack{\sigma\in\E_{\rm int},\\ \sigma=K|L}}
  \tau_\sigma\big((u_K^{k+1})^{\alpha/2}-(u_L^{k+1})^{\alpha/2}\big)
  (u_K^{k+1})^{\alpha/2-1}a_K \\
  &\phantom{xx}{}
  +\frac{\alpha\tau}{4}
  \sum_{\substack{\sigma\in\E_{\rm int},\\ \sigma=K|L}}
  \tau_\sigma\big((u_L^{k+1})^{\alpha/2}-(u_K^{k+1})^{\alpha/2}\big)
  (u_L^{k+1})^{\alpha/2-1}a_L.
\end{align*}
Replacing $a_K$ and $a_L$ by scheme \eqref{2.fv} and performing a summation
by parts, we infer that
\begin{align}
  G_\alpha(\triangle t) 
  &\le -\frac{\alpha\triangle t}{4}\sum_{K\in\TT}
  \sum_{\substack{\sigma\in\E_{\rm int},\\ \sigma=K|L}}\tau_\sigma
  \sum_{\substack{\widetilde\sigma\in\E_{\rm int},\\ \widetilde\sigma'=K|M}}
  \tau_{\widetilde\sigma}(u_K^{k+1})^{\alpha/2-1} \nonumber \\
  &\phantom{xx}{}\times\big((u_K^{k+1})^\beta-(u_M^{k+1})^\beta\big)
  \big((u_K^{k+1})^{\alpha/2}-(u_L^{k+1})^{\alpha/2}\big). \label{ineq.G}
\end{align}
Note that the expression on the right-hand side is the discrete counterpart of
the integral 
$$
  -\frac{\alpha}{2}\int_\Omega u^{\alpha/2-1}(u^\beta)_{xx}(u^{\alpha/2})_{xx}dx,
$$
appearing in \eqref{f.aux}. The condition $\alpha=2\beta$ implies immediately the 
monotonicity of $k\mapsto F_\alpha^d[u^k]$.

For the proof of the second statement, let $d=1$ and decompose the interval $\Omega$
in $N$ subintervals $K_1,\ldots,K_N$ of length $h>0$. Because of the periodic boundary
conditions, we may set $u_{N+1}^k=u_0^k$ and $u_{-1}^k=u_N^k$, where $u_i^k$ is the
approximation of the mean value of $u(\cdot,t^k)$ on the subinterval $K_i$,
$i=1,\ldots,N$.
We rewrite \eqref{ineq.G} for $\alpha=2\beta$ in one space dimension:
\begin{align*}
  G_{2\beta}(\tau) 
  &\le -\frac{\beta\tau}{2h}\sum_{i=1}^{N}
  \bigg(\sum_{j\in\{i-1,i+1\}}(u_i^{k+1})^{\beta-1}
  \big((u_i^{k+1})^\beta-({u_{j}}^{k+1})^\beta\big)\bigg)^2 \\
  &\le -\frac{\beta\tau}{2h}\min_{i=1,\ldots,N}\big((u_i^{k+1})^{2(\beta-1)}\big)
  \sum_{i=1}^{N}(z_i-z_{i-1})^2,
\end{align*}
where $z_i = (u_i^{k+1})^\beta-(u_{i+1}^{k+1})^\beta$.
The periodic boundary conditions imply that $\sum_{i=1}^N z_i=0$. Hence,
we can employ the discrete Wirtinger inequality in \cite[Theorem 1]{Shi73}
to obtain
\begin{align*}
  G_{2\beta}(\tau) 
  &\le -\frac{2\beta\tau}{h}\sin^2\frac{\pi}{N}
  \min_{i=1,\ldots,N}\big((u_i^k)^{2(\beta-1)}\big)
  \sum_{i=1}^{N}z_i^2 \\
  &= -\frac{4\beta\tau}{h}\sin^2\frac{\pi}{N}
  \min_{i=1,\ldots,N}\big((u_i^k)^{2(\beta-1)}\big)F_\alpha^d[u^{k+1}].
\end{align*}
By the discrete maximum principle, 
$\max_i(u_i^{k+1})^{2(1-\beta)}\le\max_i(u_i^0)^{2(1-\beta)}$
which is equivalent to $\min_i(u_i^{k+1})^{\beta-1}\ge\min_i(u_i^0)^{\beta-1}$. 
Therefore,
$$
  F_\alpha^d[u^{k+1}]-F_\alpha^d[u^k] 
  = G_{2\beta}(\triangle t) \le -\frac{4\beta\triangle t}{h}\sin^2\frac{\pi}{N}
  \min_{i=1,\ldots,N}\big((u_i^0)^{2(\beta-1)}\big)F_\alpha^d[u^{k+1}],
$$
and Gronwall's lemma finishes the proof.
\end{proof}


\section{Numerical experiments}\label{sec.num}

We illustrate the time decay of the solutions to the discretized 
porous-medium ($\beta=2$) and fast-diffusion equation ($\beta=1/2$) in one and two
space dimensions. 

First, let $\beta=2$. We recall that the Barenblatt profile
$$
  u_B(x,t) = (t+t_0)^{-A}\Big(C - \frac{B(\beta-1)}{2\beta}\,
  \frac{|x-x_0|^2}{(t+t_0)^{2B}}\Big)_{+}^{1/(\beta-1)}
$$
is a special solution to the porous-medium equation in the whole space. (Here, $z_+$ denotes the positive part of a function $z_+:=\max\{0,z\}$.)
The constants are given by
$$
 A = \frac{d}{d(\beta-1)+2},\quad 
 B = \frac{1}{d(\beta-1)+2},
$$
and $C$ is typically determined by the initial datum via
$\int_\Omega u(x,t)dx=\int_\Omega u(x,0)dx$.
We choose $C=B(\beta-1)(2 \beta)^{-1}(t_1+t_0)^{-2B}|x_1-x_0|^2$, where $t_1>0$
is the smallest time for which $u(x_1,t_1)=0$.

In the one-dimensional situation, we choose $\Omega=(0,1)$ with homogeneous
Neumann boundary conditions and a uniform grid
$(x_i,t^j)\in[0,1]\times[0,0.2]$ with $1\le i\le 50$ and $0\le j\le 1000$.
The initial datum is given by the Barenblatt profile $u_B(\cdot,0)$ with
$x_0=0.5$, $x_1=1$ and $t_0=0.01$. The constant $C$ is computed by using $t_1=0.1$,
which yields $C\approx 0.091$. For $0\le t\le 0.1$, the analytical solution
corresponds to the Barenblatt profile. 

The time decay of the
zeroth- and first-order entropies are depicted in Figure \ref{fig.1d.beta2}
in semi-logarithmic scale
for various values of $\alpha$. The decay rates are exponential for 
sufficiently large times, even for $\alpha>1$ (compare to Theorem \ref{thm.E.exp2})
and for $\alpha\neq 2\beta$ (see Theorem \ref{thm.F.exp2}), which indicates
that the conditions imposed in these theorems are technical.
For small times, the decay seems to be faster than the decay in the
large-time regime. This fact has been already observed in \cite[Remark 4]{CDGJ06}.
There is a significant change in the decay rate of the first-order 
entropies $F_\alpha^d$ for times around $t_1=0.1$. Indeed, the positive part of
the discrete solution, which approximates the Barenblatt profile $u_B$ for $t<t_1$, 
arrives the boundary and does not approximate $u_B$ anymore. 
The change is more apparent for $\alpha<1$.

\begin{figure}[htbp]
\centering
\includegraphics[width=80mm]{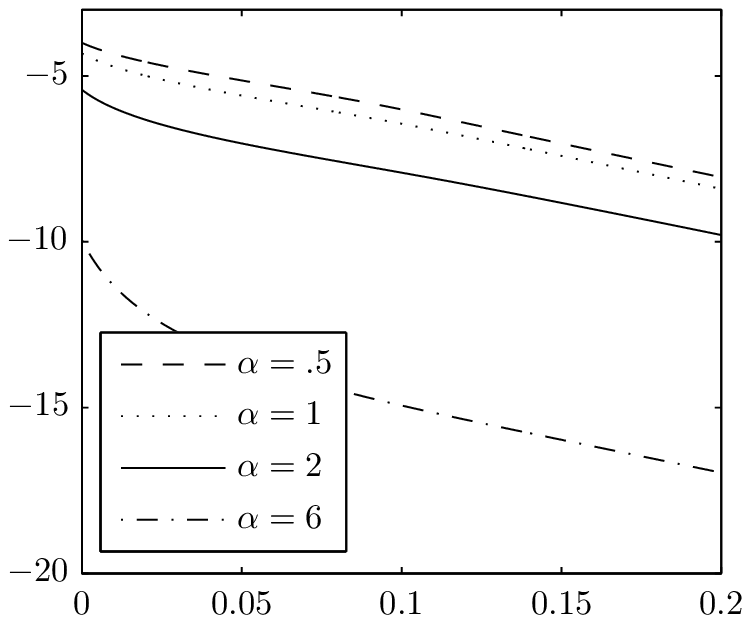}
\includegraphics[width=80mm]{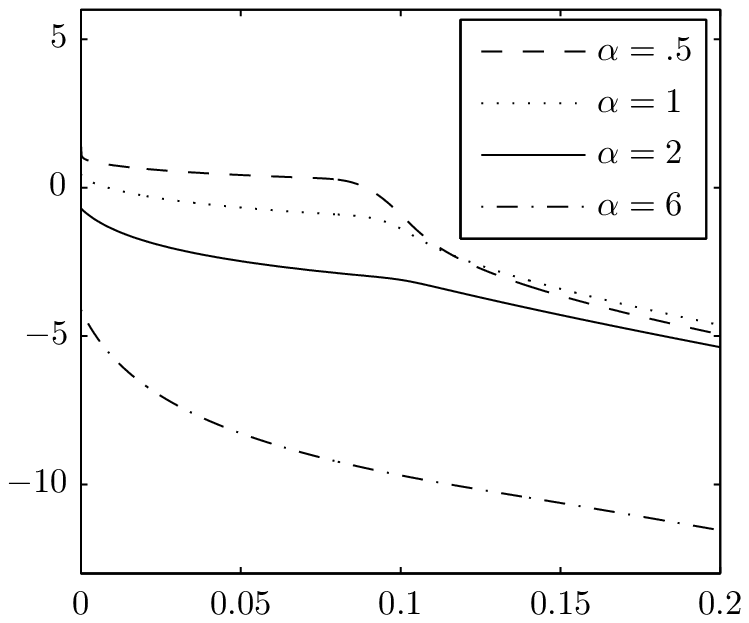}
\caption{The natural logarithm of the entropies $\log(E_\alpha^d[u](t))$ (left) and $\log(F_\alpha^d[u](t))$ (right) versus time for different values of $\alpha$ ($\beta=2$, $d=1$).}
\label{fig.1d.beta2}
\end{figure}

Next, we investigate the two-dimensional situation (still with $\beta=2$).
The domain $\Omega=(0,1)^2$ is divided into $144$ quadratic cells each of which consists
of four control volumes (see Figure \ref{fig.cells}). Again we employ
the Barenblatt profile as the initial datum, choosing $t_0=0.01$, $t_1=0.1$, and
$x_0=(0.5,0.5)$, and impose homogeneous boundary conditions. The time step size
equals $\triangle t=8\cdot 10^{-4}$. 

\begin{figure}[htbp]
  \includegraphics[width=30mm]{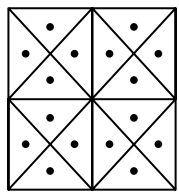}
  \caption{Four of the $144$ cells used for the two-dimensional finite-volume scheme.}
  \label{fig.cells}
\end{figure}

In Figure \ref{fig.2d.beta2}, the time evolution of the (logarithmic) 
zeroth- and first-order entropies are presented. Again, the decay seems to be
exponential for large times, even for values of $\alpha$ not covered by the
theoretical results. At time $t=t_1$, the profile reaches the boundary of the domain. 
In contrast to the one-dimensional situation, since the radially symmetric profile does not reach the boundary everywhere at the
same time, the time decay rate of $F_\alpha^d$ does not change as distinct
as in Figure \ref{fig.1d.beta2}. 

\begin{figure}[htbp]
\includegraphics[width=80mm]{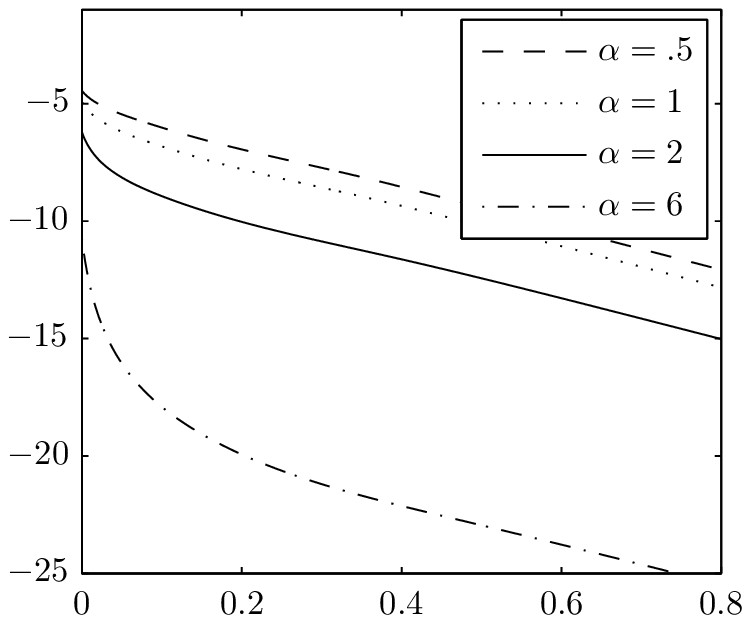}
\includegraphics[width=80mm]{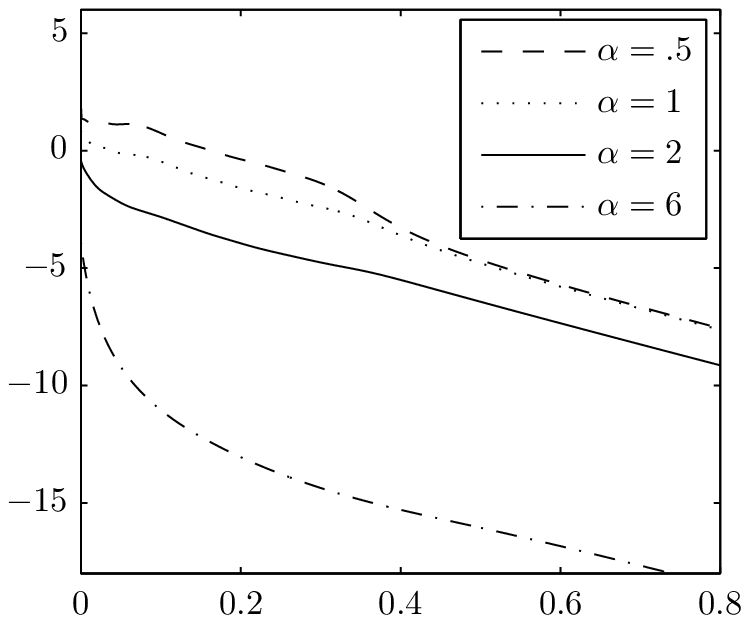}
\caption{The natural logarithm of the entropies $\log(E_\alpha^d[u](t))$ (left) and $\log(F_\alpha^d[u](t))$ (right) versus time for different values of $\alpha$ ($\beta=2$, $d=2$).}
  \label{fig.2d.beta2}
\end{figure}

Let $\beta=1/2$. The one-dimensional interval $\Omega=(0,1)$ is discretized
as before using 51 grid points and the time step size is 
$\triangle t=2\cdot 10^{-4}$. We impose homogeneous Neumann boundary conditions.
As initial datum, we choose the following truncated polynomial
$u_0(x)= C((x_0-x)(x-x_1))_{+}^2$, where $x_0=0.3$, $x_1=0.7$, and
$C=3000$. In the two-dimensional box $\Omega=(0,1)^2$, we employ
the discretization described above and the initial datum
$u_0(x)=C(R^2-|x-x_0|^2)_+^2$, where $R=0.2$, $x_0=(0.5,0.5)$ and again $C=3000$.

In the fast-diffusion case $\beta<1$, we do not expect significant changes in the
decay rate since the initial values propagate with infinite speed.
This expectation is supported by the numerical results presented in Figures
\ref{fig.1d.beta05} and \ref{fig.2d.beta05}. For a large range of values
of $\alpha$, the decay rate is exponential, at least for large times.
Interestingly, the rate seems to approach almost the same value for $\alpha\in\{0.5,1,2\}$
in Figure \ref{fig.2d.beta05}.

\begin{figure}[htbp]
\includegraphics[width=80mm]{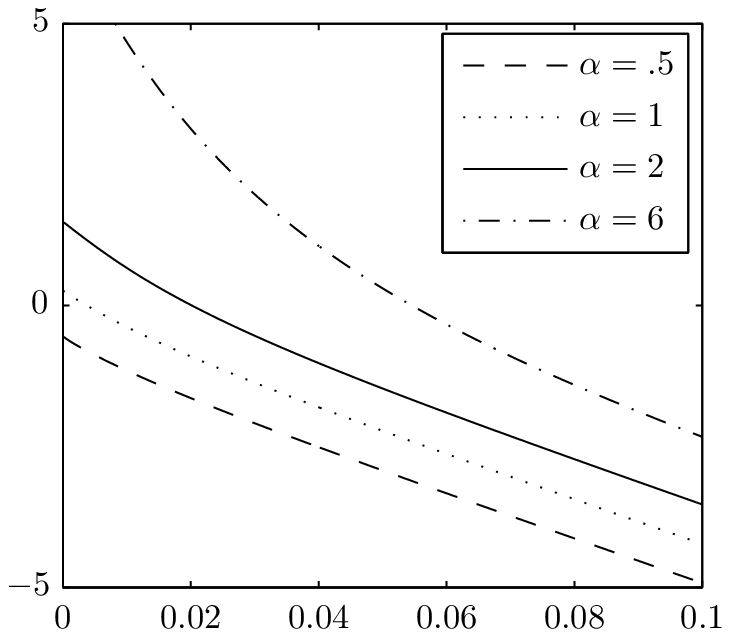}
\includegraphics[width=80mm]{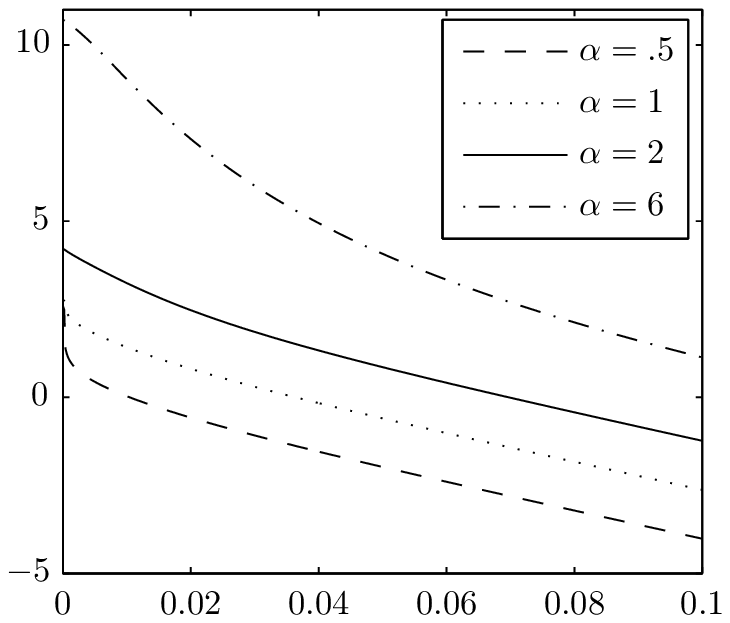}
\caption{The natural logarithm of the entropies $\log(E_\alpha^d[u](t))$ (left) and $\log(F_\alpha^d[u](t))$ (right) versus time for different values of $\alpha$ ($\beta=1/2$, $d=1$).}
\label{fig.1d.beta05}
\end{figure}

\begin{figure}[htbp]
\includegraphics[width=80mm]{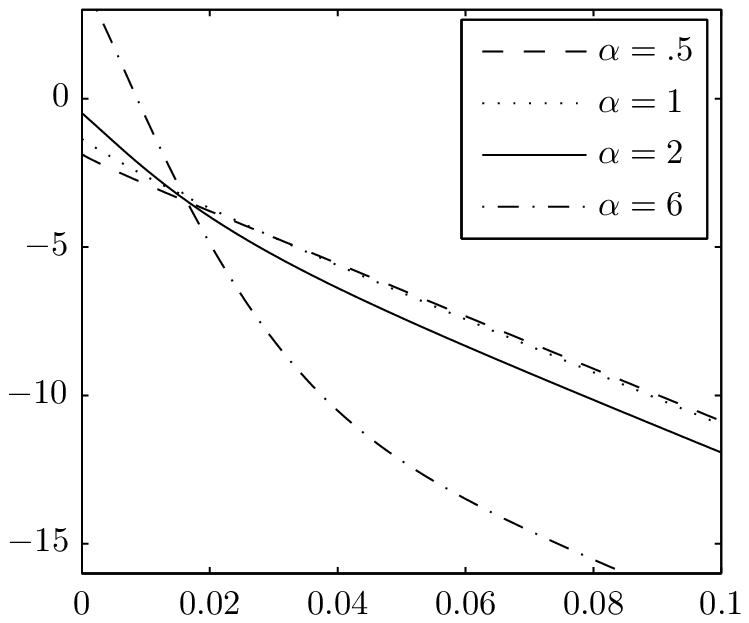}
\includegraphics[width=80mm]{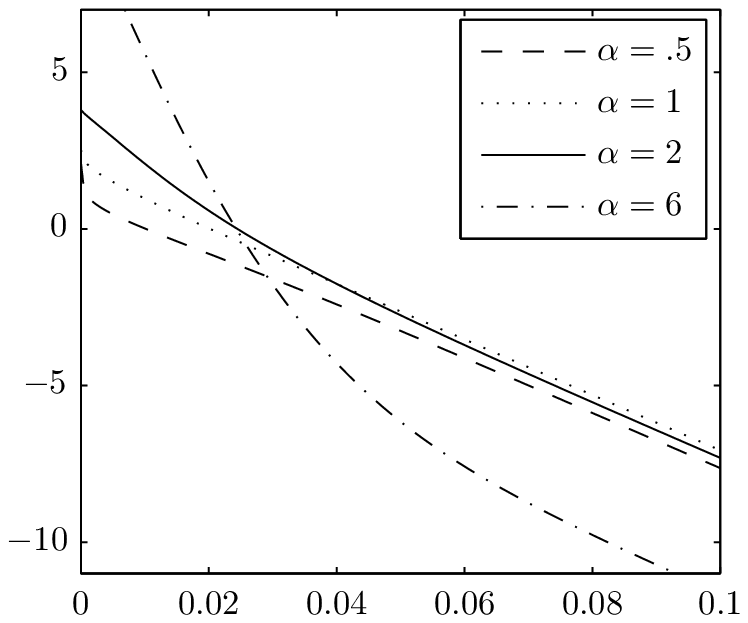}
\caption{The natural logarithm of the entropies $\log(E_\alpha^d[u](t))$ (left) and $\log(F_\alpha^d[u](t))$ (right) versus time for different values of $\alpha$ ($\beta=1/2$, $d=2$).}
\label{fig.2d.beta05}
\end{figure}


\begin{appendix}
\section{Some technical lemmas}\label{sec.app}

\subsection{Discrete Gronwall lemmas}\label{sec.app1}

First, we prove a rather general discrete nonlinear Gronwall lemma.

\begin{lemma}[Discrete nonlinear Gronwall lemma]\label{lem.gron1}
Let $f\in C^1([0,\infty))$ be a positive, nondecreasing, and convex function
such that $1/f$ is locally integrable. Define
$$
  w(x) = \int_1^x \frac{dz}{f(z)}, \quad x\ge 0.
$$
Let $(x_n)$ be a sequence of nonnegative numbers such that
$x_{n+1}-x_n+f(x_{n+1})\le 0$ for $n\in\N_0$. Then
$$
  x_n \le w^{-1}\left(w(x_0) - \frac{n}{1+f'(x_0)}\right), \quad n\in\N.
$$
\end{lemma}

Notice that the function $w$ is strictly increasing such that its inverse
is well defined. 

\begin{proof}
Since $f$ is nondecreasing and $(x_n)$ is nonincreasing, we obtain
$$
  w(x_{n+1})-w(x_n) = \int_{x_n}^{x_{n+1}}\frac{dz}{f(z)}
  \le \frac{x_{n+1}-x_n}{f(x_n)}.
$$
The sequence $(x_n)$ satisfies $f(x_{n+1})/(x_{n+1}-x_n)\ge -1$. Therefore,
\begin{align*}
  w(x_{n+1})-w(x_n) 
  &\le \Big(
  \frac{f(x_{n+1})}{x_{n+1}-x_n} + \frac{f(x_n)-f(x_{n+1})}{x_{n+1}-x_n}\Big)^{-1} \\
  &\le \Big(-1 - \frac{f(x_n)-f(x_{n+1})}{x_{n}-x_{n+1}}\Big)^{-1}.
\end{align*}
By the convexity of $f$, $f(x_n)-f(x_{n+1})\le f'(x_n)(x_n-x_{n+1})
\le f'(x_0)(x_n-x_{n+1})$, which implies that
$$
  w(x_{n+1})-w(x_n) \le (-1 - f'(x_0))^{-1}.
$$
Summing this inequality from $n=0$ to $N-1$, where $N\in\N$, yields
$$
  w(x_N) \le w(x_0) - \frac{N}{1+f'(x_0)}.
$$
Applying the inverse function of $w$ shows the lemma.
\end{proof}

The choice $f(x)=\tau Kx^\gamma$ for some $\gamma>1$ in Lemma \ref{lem.gron1}
lead to the following result.

\begin{corollary}\label{coro.gron2}
Let $(x_n)$ be a sequence of nonnegative numbers satisfying
$$
  x_{n+1}-x_n + \tau x_{n+1}^\gamma \le 0, \quad n\in\N,
$$
where $K>0$ and $\gamma>1$. Then
$$
  x_n \le \frac{1}{\big(x_0^{1-\gamma} 
  + c\tau n\big)^{1/(\gamma-1)}},
  \quad n\in\N,
$$
where $c=(\gamma-1)/(1+\gamma\tau x_0^{\gamma-1})$.
\end{corollary}


\subsection{Some inequalities}\label{sec.app2}

We show some inequalities in two variables.

\begin{lemma}\label{lem.ineq1}
Let $\alpha$, $\beta>0$. Then, for all $x$, $y\ge 0$,
\begin{equation}\label{ineq1}
  (y^\alpha-x^\alpha)(y^\beta-x^\beta) \ge \frac{4\alpha\beta}{(\alpha+\beta)^2}
  (y^{(\alpha+\beta)/2}-x^{(\alpha+\beta)/2})^2.
\end{equation}
\end{lemma}

\begin{proof}
If $y=0$, inequality \eqref{ineq1} holds. Let $y\neq 0$ and set $z=(x/y)^\beta$.
Then the inequality is proved if for all $z\ge 0$,
$$
  f(z) = (1-z^{\alpha/\beta})(1-z) - \frac{4\alpha\beta}{(\alpha+\beta)^2}
  (1-z^{(\alpha+\beta)/2\beta})^2 \ge 0.
$$
We differentiate $f$ twice:
\begin{align*}
  f'(z) &= -1 - \frac{\alpha}{\beta}z^{\alpha/\beta-1}
  + \frac{(\alpha-\beta)^2}{\beta(\alpha+\beta)}z^{\alpha/\beta}
  + \frac{4\alpha}{\alpha+\beta}z^{(\alpha+\beta)/2\beta}, \\
  f''(z) &= \frac{\alpha(\alpha-\beta)}{\beta}z^{\alpha/2\beta-3/2}
  \Big(-\frac{1}{\beta}z^{\alpha/2\beta-1/2} 
  + \frac{\alpha-\beta}{\beta(\alpha+\beta)}z^{\alpha/2\beta+1/2} 
  + \frac{2}{\alpha+\beta}\Big).
\end{align*}
Then $f(1)=0$ and $f'(1)=0$. Thus, if we show that $f$ is convex, the
assertion follows. In order to prove the convexity of $f$, we define
$$
  g(z) = -\frac{1}{\beta}z^{\alpha/2\beta-1/2} 
  + \frac{\alpha-\beta}{\beta(\alpha+\beta)}z^{\alpha/2\beta+1/2} 
  + \frac{2}{\alpha+\beta}.
$$
Then $g(1)=0$ and it holds
$$
  g'(z) = \frac{\alpha-\beta}{2\beta^2}z^{\alpha/2\beta-3/2}(-1+z),
$$
and therefore, $g'(1)=0$. Now, if $\alpha>\beta$, $g(0)=2/(\alpha+\beta)>0$,
and $g$ is decreasing in $[0,1]$ and increasing in $[1,\infty)$. Thus,
$g(z)\ge 0$ for all $z\ge 0$. If $\alpha<\beta$ then $g(0+)=-\infty$,
and $g$ is increasing in $[0,1]$ and decreasing in $[1,\infty)$. Hence,
$g(z)\le 0$ for $z\ge 0$. Independently of the sign of $\alpha-\beta$, we obtain
$$
  f''(z) = \frac{\alpha(\alpha-\beta)}{\beta}z^{\alpha/2\beta-3/2}g(z) \ge 0
$$
for all $z\ge 0$, which shows the convexity of $f$.
\end{proof}

\begin{corollary}\label{coro.ineq1}
Let $\alpha$, $\beta>0$. Then, for all $x$, $y\ge 0$,
$$
  (y^\beta-x^\beta)(y^\alpha-x^\alpha) \ge \frac{4\alpha\beta}{(\alpha+1)^2}
  \min\{x^{\beta-1},y^{\beta-1}\}(y^{(\alpha+1)/2}-x^{(\alpha+1)/2})^2.
$$
\end{corollary}

\begin{proof}
We assume without restriction that $y>x$. Then
we apply Lemma \ref{lem.ineq1} to $\beta=1$:
$$
  (y^\beta-x^\beta)(y^\alpha-x^\alpha) 
  = \frac{y^\beta-x^\beta}{y-x}(y^\alpha-x^\alpha)(y-x)
  \ge \frac{4\alpha}{(\alpha+1)^2}\frac{y^\beta-x^\beta}{y-x}
  (y^{(\alpha+1)/2}-x^{(\alpha+1)/2})^2.
$$
Since 
$$
  y^\beta - x^\beta = \beta\int_x^y t^{\beta-1}dt 
  \ge \beta\min\{x^{\beta-1},y^{\beta-1}\}(y-x),
$$
the conclusion follows.
\end{proof}

\end{appendix}


\end{document}